\documentclass{patmorin}
\usepackage[utf8]{inputenc}
\usepackage{microtype}
\usepackage{amsthm,amsmath,graphicx}
\usepackage{pat}
\usepackage[letterpaper]{hyperref}
\usepackage[table,dvipsnames]{xcolor}
\definecolor{linkblue}{named}{Blue}
\hypersetup{colorlinks=true, linkcolor=linkblue,  anchorcolor=linkblue,
citecolor=linkblue, filecolor=linkblue, menucolor=linkblue,
urlcolor=linkblue} 
\DeclareMathOperator{\x}{x}
\DeclareMathOperator{\y}{y}
\DeclareMathOperator{\survivors}{\overline\kappa}
\DeclareMathOperator{\killed}{\kappa}
\DeclareMathOperator{\tripod}{tripod}

\usepackage{array}
\newcolumntype{C}{ >{\centering\arraybackslash} m{3em} }

\newcommand{\borisspace}{\,}

\usepackage{enumitem}  
\setlist{noitemsep}

\title{\MakeUppercase{More Turán-Type Theorems for Triangles in Convex Point Sets}}

\author{%
Boris Aronov\thanks{%
    Partially supported by NSF Grants CCF-11-17336, CCF-12-18791, and
    CCF-15-40656, and by BSF grant 2014/170. Department of Computer
    Science and Engineering, Tandon School of Engineering, New York
    University, USA. },\,
Vida Dujmović\thanks{%
    Partially supported by NSERC and the Ontario Ministry of Research
    and Innovation. Department of Computer Science and Electrical Engineering,
    University of Ottawa, Canada.},
Pat Morin\thanks{%
    Partially supported by NSERC. School of Computer Science, Carleton
    University, Canada.},\,
Aurélien Ooms\thanks{%
    Supported by the Fund for Research Training in Industry and
    Agriculture (FRIA).  Département d'Informatique, Université libre 
    de Bruxelles (ULB), Belgium.  
},\,\\
and Luís Fernando Schultz Xavier da Silveira\footnotemark[2]
}

\newcommand{\taco}{\raisebox{-.1ex}{\includegraphics[height=1.6ex]{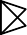}}}
\newcommand{\mariposa}{\raisebox{-.1ex}{\includegraphics[height=1.6ex]{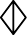}}}

\newcommand{\bat}{\raisebox{-.1ex}{\includegraphics[height=1.6ex]{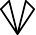}}}
\newcommand{\nested}{\raisebox{-.1ex}{\includegraphics[height=1.6ex]{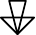}}}
\newcommand{\crossing}{\raisebox{-.1ex}{\includegraphics[height=1.6ex]{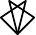}}}

\newcommand{\ears}{\raisebox{-.1ex}{\includegraphics[height=1.6ex]{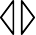}}}
\newcommand{\swords}{\raisebox{-.1ex}{\includegraphics[height=1.6ex]{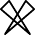}}}
\newcommand{\david}{\raisebox{-.1ex}{\includegraphics[height=1.6ex]{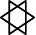}}}

\DeclareMathOperator{\ex}{ex}

\begin{document}
\begin{titlepage}
\maketitle

\begin{abstract}
  We study the following family of problems: Given a set of $n$ points
  in convex position, what is the maximum number triangles one can
  create having these points as vertices while avoiding certain sets
  of \emph{forbidden configurations}.  As forbidden configurations
  we consider all 8 ways in which a pair of triangles in such a point
  set can interact.  This leads to 256 extremal Turán-type questions.
  We give nearly tight (within a $\log n$ factor) bounds for 248 of these
  questions and show that the remaining 8 questions are all asymptotically
  equivalent to Stein's longstanding tripod packing problem.
\end{abstract}

\end{titlepage}

\tableofcontents

\newpage

\section{Introduction}
\pagenumbering{arabic}

Let $t_1$ and $t_2$ be a pair of distinct triangles whose (4 to 6)
vertices are in convex position.  There are 8 combinatorially distinct
ways that these triangles can interact:  2 ways in which the triangles
can share an edge (\taco\ and \mariposa), 3 ways in which the triangles
can share a single vertex (\bat, \nested, and \crossing), and 3 ways
in which the triangles can have no vertices in common (\ears, \swords,
and \david).  Because it is difficult to keep track of nameless entities,
we assign a mnemonic to each configuration (though the reader is encouraged
to choose their own):

\begin{center}
  \begin{tabular}{|c|c|c|c|c|c|c|c|c}\hline
    \taco & \mariposa & \bat & \nested & \crossing & \ears & \swords & \david \\
    taco & mariposa & bat & nested & crossing & ears & swords & david  \\ \hline
  \end{tabular}
\end{center}

We consider the following class of problems:  Given a set, $X$,
of combinatorial configurations of pairs of triangles, what is the
size of a largest set, $S$, of triangles one can create whose vertices are $n$
points in convex position, and such that no pair of triangles in $S$
forms a configuration in $X$?  We call the size of such a set $\ex(n,X)$.
For example, 
\begin{equation}
    \ex(n,\{\taco,\nested,\crossing,\swords,\david\}) = n-2 \borisspace .
\end{equation}
This is because the set $X=\{\taco,\nested,\crossing,\swords,\david\}$
forbids any form of crossings between the edges of triangles. Thus, the
maximum number of triangles we can have while avoiding $X$ is the number
of triangles in a triangulation of a convex $n$-gon, i.e., $n-2$. 

\subsection{Previous Work}

Since there are eight possible forbidden configurations, there are
$2^8=256$ sets $X$ for which we can study $\ex(n,X)$.  Some of these sets
have been previously studied.  Bra\ss, Rote, and Swanepoel showed that
$\ex(n,\{\ears,\swords,\bat,\nested\})\le n$ in order to solve an 
Erd\H{o}s problem
on the maximum number of maximum area/perimeter triangles determined
by a point set. Bra\ss\ \cite{brass:turan} later began a systematic
study in which he gave asymptotically tight bounds on $\ex(n,X)$ for all
singleton $X$ and all pairs $X$ of configurations in which two triangles share
a single vertex.


Of course, upper and lower bounds are inherited through the subset
relationship: $\ex(n,X) \le \ex(n,Y)$ for any $X\supseteq Y$.
\tabref{smalltable} shows the complete set of results we obtain when we
apply this exhaustively to the list of previous results.  Each entry in
this table presents the asymptotic behaviour of $\ex(n,X)$ for the set
$X$ obtained as the union of the row and column label.  Asymptotically
tight bounds are coloured green, and gaps between lower and upper bounds
are coloured red. Previous results imply 35 tight bounds for 256 of the
possible choices of $X$.  The configuration $\mariposa$ is omitted from
the table since a simple argument (\lemref{xcup}) shows that its inclusion
in $X$ does not change $\ex(n,X)$ by more than a constant factor.
Some of the results in \tabref{smalltable} are marked with an F if they are easy, or folklore. Some others are marked with H if they follow from a corresponding bound for 3-regular hypergraphs. Specifically, if $X$ includes $\{\bat,\nested,\crossing\}$ then no pair of triangles can share a vertex, so $\ex(n,X)\le n/3$ and if $X$ includes $\{\taco,\mariposa\}$ then no pair of triangles can share an edge, so $\ex(n,X)\le \binom{n}{2}$.

\begin{table}
  \centering{
    \begin{tabular}{|c@{\,}c@{\,}c@{\,}c|C|C|C|C|C|C|C|C|}\hline
&&&&\swords&\swords&\swords&\swords&&&& \\[-1mm] 
&&&&\david&&\david&&\david&&\david& \\[-1mm] 
&&&&\ears&\ears&&&\ears&\ears&& \\ 
\hline\taco&\bat&\nested&\crossing$\rule{0mm}{1em}$&\cellcolor{green!10}$1$\newline F:F&\cellcolor{green!10}$n$\newline T\ref{thm:linear-lower}:\cite{brass.rote.ea:triangles}&\cellcolor{green!10}$n$\newline T\ref{thm:linear-lower}:H&\cellcolor{green!10}$n$\newline T\ref{thm:linear-lower}:H&\cellcolor{green!10}$n$\newline T\ref{thm:linear-lower}:H&\cellcolor{green!10}$n$\newline T\ref{thm:linear-lower}:H&\cellcolor{green!10}$n$\newline T\ref{thm:linear-lower}:H&\cellcolor{green!10}$n$\newline T\ref{thm:linear-lower}:H\\ \hline
\taco&&\nested&\crossing$\rule{0mm}{1em}$&\cellcolor{red!10}$n:n^{2}$\newline T\ref{thm:linear-lower}:H&\cellcolor{red!10}$n:n^{2}$\newline T\ref{thm:linear-lower}:H&\cellcolor{red!10}$n:n^{2}$\newline T\ref{thm:linear-lower}:H&\cellcolor{red!10}$n:n^{2}$\newline T\ref{thm:linear-lower}:H&\cellcolor{red!10}$n:n^{2}$\newline T\ref{thm:linear-lower}:H&\cellcolor{red!10}$n:n^{2}$\newline T\ref{thm:linear-lower}:H&\cellcolor{red!10}$n:n^{2}$\newline T\ref{thm:linear-lower}:H&\cellcolor{red!10}$n:n^{2}$\newline T\ref{thm:linear-lower}:H\\ \hline
\taco&\bat&\nested&$\rule{0mm}{1em}$&\cellcolor{green!10}$n$\newline T\ref{thm:linear-lower}:\cite{brass.rote.ea:triangles}&\cellcolor{green!10}$n$\newline T\ref{thm:linear-lower}:\cite{brass.rote.ea:triangles}&\cellcolor{red!10}$n:n^{2}$\newline T\ref{thm:linear-lower}:H&\cellcolor{red!10}$n:n^{2}$\newline T\ref{thm:linear-lower}:H&\cellcolor{red!10}$n:n^{2}$\newline T\ref{thm:linear-lower}:H&\cellcolor{red!10}$n:n^{2}$\newline T\ref{thm:linear-lower}:H&\cellcolor{red!10}$n:n^{2}$\newline T\ref{thm:linear-lower}:H&\cellcolor{red!10}$n:n^{2}$\newline T\ref{thm:linear-lower}:H\\ \hline
\taco&&\nested&$\rule{0mm}{1em}$&\cellcolor{red!10}$n:n^{2}$\newline T\ref{thm:linear-lower}:H&\cellcolor{red!10}$n:n^{2}$\newline T\ref{thm:linear-lower}:H&\cellcolor{red!10}$n:n^{2}$\newline T\ref{thm:linear-lower}:H&\cellcolor{red!10}$n:n^{2}$\newline T\ref{thm:linear-lower}:H&\cellcolor{red!10}$n:n^{2}$\newline T\ref{thm:linear-lower}:H&\cellcolor{red!10}$n:n^{2}$\newline T\ref{thm:linear-lower}:H&\cellcolor{red!10}$n:n^{2}$\newline T\ref{thm:linear-lower}:H&\cellcolor{red!10}$n:n^{2}$\newline T\ref{thm:linear-lower}:H\\ \hline
&\bat&\nested&\crossing$\rule{0mm}{1em}$&\cellcolor{green!10}$n$\newline T\ref{thm:linear-lower}:\cite{brass.rote.ea:triangles}&\cellcolor{green!10}$n$\newline T\ref{thm:linear-lower}:\cite{brass.rote.ea:triangles}&\cellcolor{green!10}$n$\newline T\ref{thm:linear-lower}:H&\cellcolor{green!10}$n$\newline T\ref{thm:linear-lower}:H&\cellcolor{green!10}$n$\newline T\ref{thm:linear-lower}:H&\cellcolor{green!10}$n$\newline T\ref{thm:linear-lower}:H&\cellcolor{green!10}$n$\newline T\ref{thm:linear-lower}:H&\cellcolor{green!10}$n$\newline H:H\\ \hline
&&\nested&\crossing$\rule{0mm}{1em}$&\cellcolor{red!10}$n:n^{2}$\newline T\ref{thm:linear-lower}:\cite{brass:turan}&\cellcolor{red!10}$n:n^{2}$\newline T\ref{thm:linear-lower}:\cite{brass:turan}&\cellcolor{red!10}$n:n^{2}$\newline T\ref{thm:linear-lower}:\cite{brass:turan}&\cellcolor{red!10}$n:n^{2}$\newline T\ref{thm:linear-lower}:\cite{brass:turan}&\cellcolor{red!10}$n:n^{2}$\newline T\ref{thm:linear-lower}:\cite{brass:turan}&\cellcolor{red!10}$n:n^{2}$\newline T\ref{thm:linear-lower}:\cite{brass:turan}&\cellcolor{red!10}$n:n^{2}$\newline T\ref{thm:linear-lower}:\cite{brass:turan}&\cellcolor{green!10}$n^{2}$\newline \cite{brass:turan}:\cite{brass:turan}\\ \hline
&\bat&\nested&$\rule{0mm}{1em}$&\cellcolor{green!10}$n$\newline T\ref{thm:linear-lower}:\cite{brass.rote.ea:triangles}&\cellcolor{green!10}$n$\newline \cite{brass.rote.ea:triangles}:\cite{brass.rote.ea:triangles}&\cellcolor{red!10}$n:n^{2}$\newline T\ref{thm:linear-lower}:\cite{brass:turan}&\cellcolor{red!10}$n:n^{2}$\newline \cite{brass.rote.ea:triangles}:\cite{brass:turan}&\cellcolor{red!10}$n:n^{2}$\newline T\ref{thm:linear-lower}:\cite{brass:turan}&\cellcolor{red!10}$n:n^{2}$\newline \cite{brass.rote.ea:triangles}:\cite{brass:turan}&\cellcolor{red!10}$n:n^{2}$\newline T\ref{thm:linear-lower}:\cite{brass:turan}&\cellcolor{green!10}$n^{2}$\newline \cite{brass:turan}:\cite{brass:turan}\\ \hline
&&\nested&$\rule{0mm}{1em}$&\cellcolor{red!10}$n:n^{2}$\newline T\ref{thm:linear-lower}:\cite{brass:turan}&\cellcolor{red!10}$n:n^{2}$\newline \cite{brass.rote.ea:triangles}:\cite{brass:turan}&\cellcolor{red!10}$n:n^{2}$\newline T\ref{thm:linear-lower}:\cite{brass:turan}&\cellcolor{red!10}$n:n^{2}$\newline \cite{brass.rote.ea:triangles}:\cite{brass:turan}&\cellcolor{red!10}$n:n^{2}$\newline T\ref{thm:linear-lower}:\cite{brass:turan}&\cellcolor{red!10}$n:n^{2}$\newline \cite{brass.rote.ea:triangles}:\cite{brass:turan}&\cellcolor{red!10}$n:n^{2}$\newline T\ref{thm:linear-lower}:\cite{brass:turan}&\cellcolor{green!10}$n^{2}$\newline \cite{brass:turan}:\cite{brass:turan}\\ \hline
\taco&\bat&&\crossing$\rule{0mm}{1em}$&\cellcolor{red!10}$n:n^{2}$\newline T\ref{thm:linear-lower}:H&\cellcolor{red!10}$n:n^{2}$\newline T\ref{thm:linear-lower}:H&\cellcolor{red!10}$n:n^{2}$\newline T\ref{thm:linear-lower}:H&\cellcolor{red!10}$n:n^{2}$\newline T\ref{thm:linear-lower}:H&\cellcolor{red!10}$n:n^{2}$\newline T\ref{thm:linear-lower}:H&\cellcolor{red!10}$n:n^{2}$\newline T\ref{thm:linear-lower}:H&\cellcolor{red!10}$n:n^{2}$\newline T\ref{thm:linear-lower}:H&\cellcolor{red!10}$n:n^{2}$\newline T\ref{thm:linear-lower}:H\\ \hline
\taco&&&\crossing$\rule{0mm}{1em}$&\cellcolor{red!10}$n:n^{2}$\newline T\ref{thm:linear-lower}:H&\cellcolor{red!10}$n:n^{2}$\newline T\ref{thm:linear-lower}:H&\cellcolor{red!10}$n:n^{2}$\newline T\ref{thm:linear-lower}:H&\cellcolor{red!10}$n:n^{2}$\newline T\ref{thm:linear-lower}:H&\cellcolor{red!10}$n:n^{2}$\newline T\ref{thm:linear-lower}:H&\cellcolor{red!10}$n:n^{2}$\newline T\ref{thm:linear-lower}:H&\cellcolor{red!10}$n:n^{2}$\newline T\ref{thm:linear-lower}:H&\cellcolor{red!10}$n:n^{2}$\newline T\ref{thm:linear-lower}:H\\ \hline
\taco&\bat&&$\rule{0mm}{1em}$&\cellcolor{red!10}$n:n^{2}$\newline T\ref{thm:linear-lower}:H&\cellcolor{red!10}$n:n^{2}$\newline T\ref{thm:linear-lower}:H&\cellcolor{red!10}$n:n^{2}$\newline T\ref{thm:linear-lower}:H&\cellcolor{red!10}$n:n^{2}$\newline T\ref{thm:linear-lower}:H&\cellcolor{red!10}$n:n^{2}$\newline T\ref{thm:linear-lower}:H&\cellcolor{red!10}$n:n^{2}$\newline T\ref{thm:linear-lower}:H&\cellcolor{red!10}$n:n^{2}$\newline T\ref{thm:linear-lower}:H&\cellcolor{red!10}$n:n^{2}$\newline T\ref{thm:linear-lower}:H\\ \hline
\taco&&&$\rule{0mm}{1em}$&\cellcolor{red!10}$n:n^{2}$\newline T\ref{thm:linear-lower}:H&\cellcolor{red!10}$n:n^{2}$\newline T\ref{thm:linear-lower}:H&\cellcolor{red!10}$n:n^{2}$\newline T\ref{thm:linear-lower}:H&\cellcolor{red!10}$n:n^{2}$\newline T\ref{thm:linear-lower}:H&\cellcolor{red!10}$n:n^{2}$\newline T\ref{thm:linear-lower}:H&\cellcolor{red!10}$n:n^{2}$\newline T\ref{thm:linear-lower}:H&\cellcolor{red!10}$n:n^{2}$\newline T\ref{thm:linear-lower}:H&\cellcolor{green!10}$n^{2}$\newline H:H\\ \hline
&\bat&&\crossing$\rule{0mm}{1em}$&\cellcolor{red!10}$n:n^{2}$\newline T\ref{thm:linear-lower}:\cite{brass:turan}&\cellcolor{red!10}$n:n^{2}$\newline T\ref{thm:linear-lower}:\cite{brass:turan}&\cellcolor{red!10}$n:n^{2}$\newline T\ref{thm:linear-lower}:\cite{brass:turan}&\cellcolor{red!10}$n:n^{2}$\newline T\ref{thm:linear-lower}:\cite{brass:turan}&\cellcolor{red!10}$n:n^{2}$\newline T\ref{thm:linear-lower}:\cite{brass:turan}&\cellcolor{red!10}$n:n^{2}$\newline T\ref{thm:linear-lower}:\cite{brass:turan}&\cellcolor{red!10}$n:n^{2}$\newline T\ref{thm:linear-lower}:\cite{brass:turan}&\cellcolor{green!10}$n^{2}$\newline \cite{brass:turan}:\cite{brass:turan}\\ \hline
&&&\crossing$\rule{0mm}{1em}$&\cellcolor{red!10}$n:n^{2}$\newline T\ref{thm:linear-lower}:\cite{brass:turan}&\cellcolor{red!10}$n:n^{2}$\newline T\ref{thm:linear-lower}:\cite{brass:turan}&\cellcolor{red!10}$n:n^{2}$\newline T\ref{thm:linear-lower}:\cite{brass:turan}&\cellcolor{red!10}$n:n^{2}$\newline T\ref{thm:linear-lower}:\cite{brass:turan}&\cellcolor{red!10}$n:n^{2}$\newline T\ref{thm:linear-lower}:\cite{brass:turan}&\cellcolor{red!10}$n:n^{2}$\newline T\ref{thm:linear-lower}:\cite{brass:turan}&\cellcolor{red!10}$n:n^{2}$\newline T\ref{thm:linear-lower}:\cite{brass:turan}&\cellcolor{green!10}$n^{2}$\newline \cite{brass:turan}:\cite{brass:turan}\\ \hline
&\bat&&$\rule{0mm}{1em}$&\cellcolor{red!10}$n:n^{2}$\newline T\ref{thm:linear-lower}:\cite{brass:turan}&\cellcolor{red!10}$n:n^{2}$\newline \cite{brass.rote.ea:triangles}:\cite{brass:turan}&\cellcolor{red!10}$n:n^{2}$\newline T\ref{thm:linear-lower}:\cite{brass:turan}&\cellcolor{red!10}$n:n^{2}$\newline \cite{brass.rote.ea:triangles}:\cite{brass:turan}&\cellcolor{red!10}$n:n^{2}$\newline T\ref{thm:linear-lower}:\cite{brass:turan}&\cellcolor{red!10}$n:n^{3}$\newline \cite{brass.rote.ea:triangles}:F&\cellcolor{red!10}$n:n^{2}$\newline T\ref{thm:linear-lower}:\cite{brass:turan}&\cellcolor{green!10}$n^{3}$\newline \cite{brass:turan}:\cite{brass:turan}\\ \hline
&&&$\rule{0mm}{1em}$&\cellcolor{green!10}$n^{2}$\newline H:H&\cellcolor{green!10}$n^{2}$\newline H:\cite{brass:turan}&\cellcolor{green!10}$n^{2}$\newline H:\cite{brass:turan}&\cellcolor{green!10}$n^{2}$\newline \cite{brass:turan}:\cite{brass:turan}&\cellcolor{green!10}$n^{2}$\newline H:\cite{brass:turan}&\cellcolor{green!10}$n^{3}$\newline \cite{brass:turan}:\cite{brass:turan}&\cellcolor{green!10}$n^{2}$\newline \cite{brass:turan}:\cite{brass:turan}&\cellcolor{green!10}$n^{3}$\newline F:F\\ \hline
\end{tabular}
  }
	\caption{Previous lower and upper bounds for $\ex(n,X)$. T2 denotes a (easy) lower bound of $\Omega(n)$ that appears in \thmref{linear-lower}.  F denotes an obvious, or folklore result. H denotes a bound that follows from the corresponding bound on 3-regular hypergraphs.}  \tablabel{smalltable}
\end{table}

\subsection{New Results}

In the current paper, we determine, up to a logarithmic factor, the
asymptotics of $\ex(n,X)$ for 248 sets $X$.  These results are shown in
\tabref{bigtable}. For the remaining 8 sets, we have determined that
the asymptotics are all the same and are equivalent to a problem that
appears in various contexts and under different names, including monotone
matrices, tripod packing, and 2-comparable triples. We discuss this
problem and its rich history in \secref{tripods}.  

\begin{table}
  \centering{
    \begin{tabular}{|c@{\,}c@{\,}c@{\,}c|C|C|C|C|C|C|C|C|}\hline
&&&&\swords&\swords&\swords&\swords&&&& \\[-1mm] 
&&&&\david&&\david&&\david&&\david& \\[-1mm] 
&&&&\ears&\ears&&&\ears&\ears&& \\ 
\hline\taco&\bat&\nested&\crossing$\rule{0mm}{1em}$&\cellcolor{green!10}$1$\newline F:F&\cellcolor{green!10}$n$\newline T\ref{thm:linear-lower}:\cite{brass.rote.ea:triangles}&\cellcolor{green!10}$n$\newline T\ref{thm:linear-lower}:H&\cellcolor{green!10}$n$\newline T\ref{thm:linear-lower}:H&\cellcolor{green!10}$n$\newline T\ref{thm:linear-lower}:H&\cellcolor{green!10}$n$\newline T\ref{thm:linear-lower}:H&\cellcolor{green!10}$n$\newline T\ref{thm:linear-lower}:H&\cellcolor{green!10}$n$\newline T\ref{thm:linear-lower}:H\\ \hline
\taco&&\nested&\crossing$\rule{0mm}{1em}$&\cellcolor{green!40}$n^*$\newline T\ref{thm:linear-lower}:T\ref{thm:crossing-swords}&\cellcolor{green!40}$n^*$\newline T\ref{thm:linear-lower}:T\ref{thm:crossing-swords}&\cellcolor{green!40}$n^*$\newline T\ref{thm:linear-lower}:T\ref{thm:crossing-swords}&\cellcolor{green!40}$n^*$\newline T\ref{thm:linear-lower}:T\ref{thm:crossing-swords}&\cellcolor{green!40}$n^*$\newline T\ref{thm:linear-lower}:T\ref{thm:taco-nested-crossing}&\cellcolor{green!40}$n^*$\newline T\ref{thm:linear-lower}:T\ref{thm:taco-nested-crossing}&\cellcolor{green!40}$n^*$\newline T\ref{thm:linear-lower}:T\ref{thm:taco-nested-crossing}&\cellcolor{green!40}$n^*$\newline T\ref{thm:linear-lower}:\textbf{T\ref{thm:taco-nested-crossing}}\\ \hline
\taco&\bat&\nested&$\rule{0mm}{1em}$&\cellcolor{green!10}$n$\newline T\ref{thm:linear-lower}:\cite{brass.rote.ea:triangles}&\cellcolor{green!10}$n$\newline T\ref{thm:linear-lower}:\cite{brass.rote.ea:triangles}&\cellcolor{green!40}$n^*$\newline T\ref{thm:linear-lower}:T\ref{thm:taco-swords}&\cellcolor{green!40}$n^*$\newline T\ref{thm:linear-lower}:T\ref{thm:taco-swords}&\cellcolor{green!40}$n^*$\newline T\ref{thm:linear-lower}:T\ref{thm:nested-bat-david}&\cellcolor{red!40}tripods&\cellcolor{green!40}$n^*$\newline T\ref{thm:linear-lower}:T\ref{thm:nested-bat-david}&\cellcolor{red!40}tripods\\ \hline
\taco&&\nested&$\rule{0mm}{1em}$&\cellcolor{green!40}$n^*$\newline T\ref{thm:linear-lower}:T\ref{thm:taco-swords}&\cellcolor{green!40}$n^*$\newline T\ref{thm:linear-lower}:T\ref{thm:taco-swords}&\cellcolor{green!40}$n^*$\newline T\ref{thm:linear-lower}:T\ref{thm:taco-swords}&\cellcolor{green!40}$n^*$\newline T\ref{thm:linear-lower}:T\ref{thm:taco-swords}&\cellcolor{green!40}$n^*$\newline T\ref{thm:linear-lower}:T\ref{thm:taco-nested-david}&\cellcolor{red!40}tripods&\cellcolor{green!40}$n^*$\newline T\ref{thm:linear-lower}:\textbf{T\ref{thm:taco-nested-david}}&\cellcolor{red!40}tripods\\ \hline
&\bat&\nested&\crossing$\rule{0mm}{1em}$&\cellcolor{green!10}$n$\newline T\ref{thm:linear-lower}:\cite{brass.rote.ea:triangles}&\cellcolor{green!10}$n$\newline T\ref{thm:linear-lower}:\cite{brass.rote.ea:triangles}&\cellcolor{green!10}$n$\newline T\ref{thm:linear-lower}:H&\cellcolor{green!10}$n$\newline T\ref{thm:linear-lower}:H&\cellcolor{green!10}$n$\newline T\ref{thm:linear-lower}:H&\cellcolor{green!10}$n$\newline T\ref{thm:linear-lower}:H&\cellcolor{green!10}$n$\newline T\ref{thm:linear-lower}:H&\cellcolor{green!10}$n$\newline H:H\\ \hline
&&\nested&\crossing$\rule{0mm}{1em}$&\cellcolor{green!40}$n^*$\newline T\ref{thm:linear-lower}:T\ref{thm:crossing-swords}&\cellcolor{green!40}$n^*$\newline T\ref{thm:linear-lower}:T\ref{thm:crossing-swords}&\cellcolor{green!40}$n^*$\newline T\ref{thm:linear-lower}:T\ref{thm:crossing-swords}&\cellcolor{green!40}$n^*$\newline T\ref{thm:linear-lower}:T\ref{thm:crossing-swords}&\cellcolor{green!40}$n^*$\newline T\ref{thm:linear-lower}:T\ref{thm:nested-crossing-ears}&\cellcolor{green!40}$n^*$\newline T\ref{thm:linear-lower}:\textbf{T\ref{thm:nested-crossing-ears}}&\cellcolor{green!40}$n^{2}$\newline \textbf{T\ref{thm:david-nested-crossing}}:\cite{brass:turan}&\cellcolor{green!10}$n^{2}$\newline \cite{brass:turan}:\cite{brass:turan}\\ \hline
&\bat&\nested&$\rule{0mm}{1em}$&\cellcolor{green!10}$n$\newline T\ref{thm:linear-lower}:\cite{brass.rote.ea:triangles}&\cellcolor{green!10}$n$\newline \cite{brass.rote.ea:triangles}:\cite{brass.rote.ea:triangles}&\cellcolor{green!40}$n^*$\newline T\ref{thm:linear-lower}:T\ref{thm:nested-swords}&\cellcolor{green!40}$n^*$\newline \cite{brass.rote.ea:triangles}:T\ref{thm:nested-swords}&\cellcolor{green!40}$n^*$\newline T\ref{thm:linear-lower}:T\ref{thm:nested-bat-david}&\cellcolor{green!40}$n^{2}$\newline \textbf{T\ref{thm:bat-nested-ears}}:\cite{brass:turan}&\cellcolor{green!40}$n^*$\newline T\ref{thm:linear-lower}:\textbf{T\ref{thm:nested-bat-david}}&\cellcolor{green!10}$n^{2}$\newline \cite{brass:turan}:\cite{brass:turan}\\ \hline
&&\nested&$\rule{0mm}{1em}$&\cellcolor{green!40}$n^*$\newline T\ref{thm:linear-lower}:T\ref{thm:nested-swords}&\cellcolor{green!40}$n^*$\newline \cite{brass.rote.ea:triangles}:T\ref{thm:nested-swords}&\cellcolor{green!40}$n^*$\newline T\ref{thm:linear-lower}:T\ref{thm:nested-swords}&\cellcolor{green!40}$n^*$\newline \cite{brass.rote.ea:triangles}:\textbf{T\ref{thm:nested-swords}}&\cellcolor{green!40}$n^*$\newline T\ref{thm:linear-lower}:\textbf{T\ref{thm:nested-ears-david}}&\cellcolor{green!40}$n^{2}$\newline T\ref{thm:bat-nested-ears}:\cite{brass:turan}&\cellcolor{green!40}$n^{2}$\newline T\ref{thm:david-nested-crossing}:\cite{brass:turan}&\cellcolor{green!10}$n^{2}$\newline \cite{brass:turan}:\cite{brass:turan}\\ \hline
\taco&\bat&&\crossing$\rule{0mm}{1em}$&\cellcolor{green!40}$n^*$\newline T\ref{thm:linear-lower}:T\ref{thm:crossing-swords}&\cellcolor{green!40}$n^*$\newline T\ref{thm:linear-lower}:T\ref{thm:crossing-swords}&\cellcolor{green!40}$n^*$\newline T\ref{thm:linear-lower}:T\ref{thm:crossing-swords}&\cellcolor{green!40}$n^*$\newline T\ref{thm:linear-lower}:T\ref{thm:crossing-swords}&\cellcolor{green!40}$n^{2}$\newline \textbf{T\ref{dilwad}}:H&\cellcolor{green!40}$n^{2}$\newline T\ref{dilwad}:H&\cellcolor{green!40}$n^{2}$\newline T\ref{dilwad}:H&\cellcolor{green!40}$n^{2}$\newline T\ref{dilwad}:H\\ \hline
\taco&&&\crossing$\rule{0mm}{1em}$&\cellcolor{green!40}$n^*$\newline T\ref{thm:linear-lower}:T\ref{thm:crossing-swords}&\cellcolor{green!40}$n^*$\newline T\ref{thm:linear-lower}:T\ref{thm:crossing-swords}&\cellcolor{green!40}$n^*$\newline T\ref{thm:linear-lower}:T\ref{thm:crossing-swords}&\cellcolor{green!40}$n^*$\newline T\ref{thm:linear-lower}:T\ref{thm:crossing-swords}&\cellcolor{green!40}$n^{2}$\newline T\ref{dilwad}:H&\cellcolor{green!40}$n^{2}$\newline T\ref{dilwad}:H&\cellcolor{green!40}$n^{2}$\newline T\ref{dilwad}:H&\cellcolor{green!40}$n^{2}$\newline T\ref{dilwad}:H\\ \hline
\taco&\bat&&$\rule{0mm}{1em}$&\cellcolor{green!40}$n^*$\newline T\ref{thm:linear-lower}:T\ref{thm:taco-swords}&\cellcolor{green!40}$n^*$\newline T\ref{thm:linear-lower}:T\ref{thm:taco-swords}&\cellcolor{green!40}$n^*$\newline T\ref{thm:linear-lower}:T\ref{thm:taco-swords}&\cellcolor{green!40}$n^*$\newline T\ref{thm:linear-lower}:T\ref{thm:taco-swords}&\cellcolor{green!40}$n^{2}$\newline T\ref{dilwad}:H&\cellcolor{green!40}$n^{2}$\newline T\ref{dilwad}:H&\cellcolor{green!40}$n^{2}$\newline T\ref{dilwad}:H&\cellcolor{green!40}$n^{2}$\newline T\ref{dilwad}:H\\ \hline
\taco&&&$\rule{0mm}{1em}$&\cellcolor{green!40}$n^*$\newline T\ref{thm:linear-lower}:T\ref{thm:taco-swords}&\cellcolor{green!40}$n^*$\newline T\ref{thm:linear-lower}:T\ref{thm:taco-swords}&\cellcolor{green!40}$n^*$\newline T\ref{thm:linear-lower}:T\ref{thm:taco-swords}&\cellcolor{green!40}$n^*$\newline T\ref{thm:linear-lower}:\textbf{T\ref{thm:taco-swords}}&\cellcolor{green!40}$n^{2}$\newline T\ref{dilwad}:H&\cellcolor{green!40}$n^{2}$\newline T\ref{dilwad}:H&\cellcolor{green!40}$n^{2}$\newline T\ref{dilwad}:H&\cellcolor{green!10}$n^{2}$\newline H:H\\ \hline
&\bat&&\crossing$\rule{0mm}{1em}$&\cellcolor{green!40}$n^*$\newline T\ref{thm:linear-lower}:T\ref{thm:crossing-swords}&\cellcolor{green!40}$n^*$\newline T\ref{thm:linear-lower}:T\ref{thm:crossing-swords}&\cellcolor{green!40}$n^*$\newline T\ref{thm:linear-lower}:T\ref{thm:crossing-swords}&\cellcolor{green!40}$n^*$\newline T\ref{thm:linear-lower}:T\ref{thm:crossing-swords}&\cellcolor{green!40}$n^{2}$\newline T\ref{dilwad}:\cite{brass:turan}&\cellcolor{green!40}$n^{2}$\newline T\ref{dilwad}:\cite{brass:turan}&\cellcolor{green!40}$n^{2}$\newline T\ref{dilwad}:\cite{brass:turan}&\cellcolor{green!10}$n^{2}$\newline \cite{brass:turan}:\cite{brass:turan}\\ \hline
&&&\crossing$\rule{0mm}{1em}$&\cellcolor{green!40}$n^*$\newline T\ref{thm:linear-lower}:T\ref{thm:crossing-swords}&\cellcolor{green!40}$n^*$\newline T\ref{thm:linear-lower}:T\ref{thm:crossing-swords}&\cellcolor{green!40}$n^*$\newline T\ref{thm:linear-lower}:T\ref{thm:crossing-swords}&\cellcolor{green!40}$n^*$\newline T\ref{thm:linear-lower}:\textbf{T\ref{thm:crossing-swords}}&\cellcolor{green!40}$n^{2}$\newline T\ref{dilwad}:\cite{brass:turan}&\cellcolor{green!40}$n^{2}$\newline T\ref{dilwad}:\cite{brass:turan}&\cellcolor{green!40}$n^{2}$\newline T\ref{thm:david-nested-crossing}:\cite{brass:turan}&\cellcolor{green!10}$n^{2}$\newline \cite{brass:turan}:\cite{brass:turan}\\ \hline
&\bat&&$\rule{0mm}{1em}$&\cellcolor{green!40}$n^{2}$\newline \textbf{T\ref{thm:swords-bat-ears-david}}:\cite{brass:turan}&\cellcolor{green!40}$n^{2}$\newline T\ref{thm:swords-bat-ears-david}:\cite{brass:turan}&\cellcolor{green!40}$n^{2}$\newline T\ref{thm:swords-bat-ears-david}:\cite{brass:turan}&\cellcolor{green!40}$n^{2}$\newline T\ref{thm:swords-bat-ears-david}:\cite{brass:turan}&\cellcolor{green!40}$n^{2}$\newline T\ref{thm:swords-bat-ears-david}:\cite{brass:turan}&\cellcolor{green!40}$n^{3}$\newline \textbf{T\ref{thm:pairwise-crossing}}:F&\cellcolor{green!40}$n^{2}$\newline T\ref{thm:swords-bat-ears-david}:\cite{brass:turan}&\cellcolor{green!10}$n^{3}$\newline \cite{brass:turan}:\cite{brass:turan}\\ \hline
&&&$\rule{0mm}{1em}$&\cellcolor{green!10}$n^{2}$\newline H:H&\cellcolor{green!10}$n^{2}$\newline H:\cite{brass:turan}&\cellcolor{green!10}$n^{2}$\newline H:\cite{brass:turan}&\cellcolor{green!10}$n^{2}$\newline \cite{brass:turan}:\cite{brass:turan}&\cellcolor{green!10}$n^{2}$\newline H:\cite{brass:turan}&\cellcolor{green!10}$n^{3}$\newline \cite{brass:turan}:\cite{brass:turan}&\cellcolor{green!10}$n^{2}$\newline \cite{brass:turan}:\cite{brass:turan}&\cellcolor{green!10}$n^{3}$\newline F:F\\ \hline
\end{tabular}
    \\[1ex]
    $n^*=n:n\log n$ \quad $\mathrm{tripods}=n^{1.546}:n^2/e^{\Omega(\log^* n)}$
  }
  \caption{New and previous bounds for $\ex(n,X)$, up to a factor of $\log n$.
  New near-optimal results are in dark(er) green. TX denotes Theorem~X in this paper and [X] denotes reference X in this paper. For example T16:\cite{brass:turan} denotes a lower bound that appears in Theorem~16 and an upper bound due to Bra\ss\ \cite{brass:turan}.}
  \tablabel{bigtable}
\end{table}

The rest of this paper is organized as follows.  In
\secref{easy-results} we present a few easy results that we need for
completeness.  In \secref{points-of-view} we discuss different ways
of thinking about the problem.  In particular, we present a series of
puzzles whose solutions determine the asymptotic growth of $\ex(n,X)$.
In \secref{new-results}, which represents the technical meat of the paper,
we use these puzzles to derive new upper and lower bounds.

\section{Easy Results}
\seclabel{easy-results}

In this section we present an easy result that cuts our work in half
by reducing the number of problems from 256 to 128.  We then describe some
easy lower bound constructions that are required for completeness.

\subsection{Mariposas are Irrelevant}

The following lemma shows that including the $\mariposa$ configuration in
the set $X$ of forbidden configurations has no effect on the asymptotics
of $\ex(n,X)$.

\begin{lem}\lemlabel{xcup}
   For any $X$, $\ex(n,X\cup\{\mariposa\}) \ge \ex(n,X)/8$.
\end{lem}

\begin{proof}
  Let $S$ be a set of triangles that achieves $\ex(n,X)$. For each pair
  of vertices $u$ and $w$ independently and uniformly choose a direction
  $\overrightarrow{uw}$ or $\overleftarrow{uw}$.  We then obtain a set
  $S'\subseteq S$ by removing any triangle that has a directed edge for
  which the triangle is to the left of the edge.  Observe that the set
  $S'$ does not contain a $\mariposa$ configuration.

  For any particular triangle $t\in S$, the probability that $t\in S'$
  is exactly $1/8$ since each of $t$'s three edges must be directed
  clockwise and edge directions are chosen independently.  By linearity of
  expectation, $\E[|S'|]=|S|/8=\ex(n,X)/8$.  We conclude therefore that
  there exists some subset $S''\subseteq S$ of size least $\ex(n,X)/8$
  that does not contain a $\mariposa$ configuration.  The set $S''$ proves
  that $\ex(n,X\cup\{\mariposa\}) \ge \ex(n,X)/8$.
\end{proof}

\subsection{Cubic-Sized Sets of Pairwise Crossing Triangles}

\begin{thm}\thmlabel{pairwise-crossing}
   $\ex(n,\{\ears,\bat,\mariposa\})\in\Omega(n^3)$.
\end{thm}

\begin{proof}
  Partition the vertices of the convex $n$-gon into three contiguous
  sets, $A$, $B$, and $C$, each of size $\lfloor n/3\rfloor$ or
  $\lceil n/3\rceil$, as appropriate.  Consider the set, $S$, of all
  triangles having one vertex in each of $A$, $B$, and $C$. It is easy
  to check that any two triangles in $S$ have a pair of edges that
  cross, thus they do not form any of \ears, \bat, or \mariposa.
  Furthermore, $|S|\ge \lfloor n/3\rfloor^3\in\Omega(n^3)$, so
  $\ex(n,\{\ears,\bat,\mariposa\})\in\Omega(n^3)$.
\end{proof}

\subsection{Linear-Sized Sets Using Only a Single Configuration}

Since it is not explicitly stated in previous work, and we need it to complete our table, we now observe that for any configuration $x\in\{\taco,\bat,\nested,\crossing,\ears,\swords,\david\}$, one can create a linear-sized set of triangles that avoids all configurations except $x$.

\begin{thm}\thmlabel{linear-lower}
For any $X\subsetneq\{\taco,\bat,\nested,\crossing,\ears,\swords,\david\}$,
$\ex(n,X)\in \Omega(n)$.
\end{thm}

\begin{proof}
  Let $x\in\{\taco,\bat,\nested,\crossing,\ears,\swords,\david\}$ be
  a configuration not in $X$.  Label the vertices of our convex $n$-gon
  $1,\ldots,n$ in counterclockwise order.  Depending on the value of $x$,
  we use one of the following constructions (see \figref{linear-lower}):
  \begin{figure}
    \centering{
      \begin{tabular}{ccccccc}
       \includegraphics{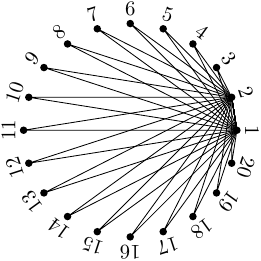} &
       \includegraphics{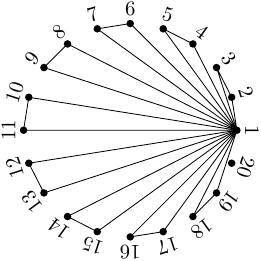} &
       \includegraphics{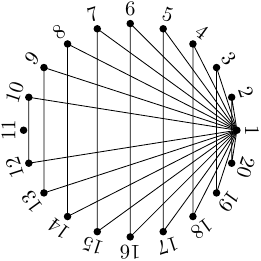} \\
       \taco & \bat & \nested \\[1em]
       \includegraphics{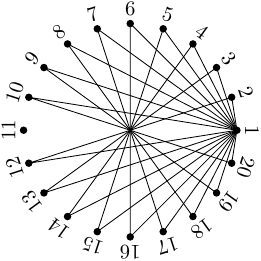} & 
       \includegraphics{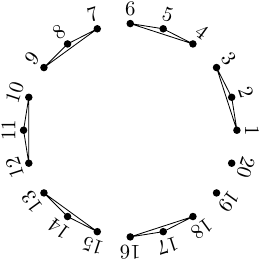} &
       \includegraphics{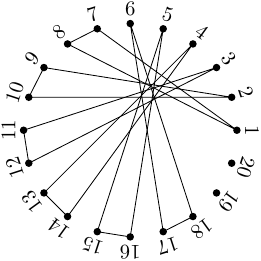} \\
       \crossing & \ears & \swords \\[1em]
       &
       \includegraphics{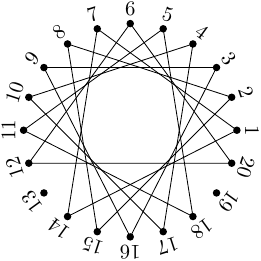} \\
       & \david
      \end{tabular}
    }
    \caption{Constructions used in the proof of \thmref{linear-lower}.}
    \figlabel{linear-lower}
  \end{figure}
 \begin{enumerate}
    \item For $x=\taco$, we use the set of triangles $\{(1,2,i):i\in\{3,\ldots,n\}\}$.
    \item For $x=\bat$, we use the set of triangles $\{(1,2i,2i+1): i\in\{1,\ldots,\lfloor n/2\rfloor-1 \}\}$.
    \item For $x=\nested$, we use the set of triangles $\{(1,i,n+2-i): i\in\{2, \ldots, \lfloor n/2\rfloor\}\}$.
    \item For $x=\crossing$, we use the set of triangles $\{(1,i,\lfloor n/2\rfloor+i): i\in\{2,\ldots, \lfloor n/2\rfloor\}\}$.
    \item For $x=\ears$, we use the set of triangles $\{(3i-2,3i-1,3i): i\in\{1,\ldots, \lfloor n/3\rfloor\}\}$.
    \item For $x=\swords$, we use the set of triangles $\{(i,\lfloor n/3\rfloor+2i-1,\lfloor n/3\rfloor+2i): i\in\{1,\ldots, \lfloor n/3\rfloor\}\}$.
    \item For $x=\david$, we use the set of triangles $\{(i,\lfloor n/3\rfloor+i,\lfloor 2n/3\rfloor+i): i\in\{1,\ldots, \lfloor n/3\rfloor\}\}$.
 \end{enumerate}
 In each case, the size of the set is $\Omega(n)$ and it is
 straightforward to verify that each pair of triangles in the set forms
 the configuration $x$ and therefore avoids all configurations in~$X$.
\end{proof}

\section{Points of View}
\seclabel{points-of-view}

In this section we describe different, but equivalent (up to a logarithmic factor), views of the problem.  One of these (the dot puzzle view) will be our main line of attack for the most difficult cases.

\subsection{The Top/Bottom View}

It will be helpful to gain a sense of orientation by considering
a top/bottom variant of $\ex(n,X)$ that is defined as follows (see
\figref{top-bottom}).  Partition the vertices of a convex $n$-gon using
a horizontal line into a \emph{top half} of size $\lceil n/2\rceil$
and a \emph{bottom half} of size $\lfloor n/2\rfloor$.  We define
$\ex'(n,X)$ analogously to $\ex(n,X)$ except that we only count triangles
having one vertex in the bottom half and two vertices in the top half.
When studying $\ex'$, each triangle we count has a naturally defined
\emph{bottom vertex} in the bottom half and a \emph{left vertex} and
\emph{right vertex}, each in the top half.

\begin{figure}
  \centering{
    \includegraphics{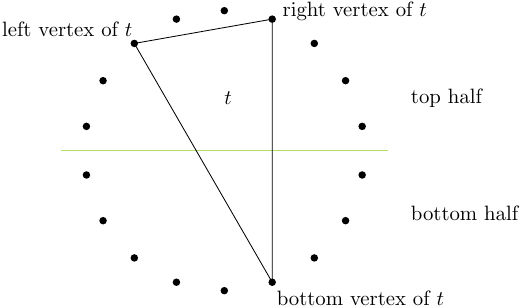}
  }
  \caption{$\ex'$ only counts triangles with two vertices in the top half
     and one vertex in the bottom half.}
  \figlabel{top-bottom}
\end{figure}

Clearly $\ex(n,X)\ge\ex'(n,X)$.  The following lemma shows that, without
losing much precision, we can also upper bound $\ex(n,X)$ by $\ex'(n,X)$.

\begin{lem}\lemlabel{top-bottom}
  If $\ex'(n,X)\in O(n^c)$, then
  \[
     \ex(n,X)\in 
        \begin{cases} 
            O(n^c)     & \text{if $c>1$,} \\
            O(n\log n) & \text{if $c=1$.}
        \end{cases}
  \]
\end{lem}

\begin{proof}
   Let $S$ be a set of triangles that avoids $X$.  Every triangle in $S$
   is of one of the following types:
   \begin{enumerate}
      \item It has one vertex in the top half and two in the bottom half;
        there are $O(n^{c})$ such triangles.
      \item It has two vertices in the top half and one in the bottom
        half; there are $O(n^{c})$ such triangles.
      \item It has all three vertices in the top half; there are at most
        $\ex(\lceil n/2\rceil,X)$ such triangles.
      \item It has all three vertices in the bottom half; there are at
        most $\ex(\lfloor n/2\rfloor,X)$ such triangles.
   \end{enumerate}
   Thus, we obtain the recurrence inequality
   \[  \ex(n,X) \le O(n^{c}) + \ex(\lceil n/2\rceil,X) + \ex(\lfloor n/2\rfloor,X) \borisspace ,\]
   which solves to $O(n^c)$ for $c>1$ and $O(n\log n)$ for $c=1$ \cite[Section~4.3]{cormen.leiserson.ea:introduction}.
\end{proof}

\subsection{The Dot Puzzle View}

The top-bottom version of the problem gives us a sense of orientation,
but it is still difficult to visualize the sets of triangles obtained
this way. Next, we show that there is a corresponding puzzle that is
easy to visualize.  Refer to \figref{point-view}.  

In this puzzle, we are given $\binom{n}{2}$ points,
\[
    Q = \{(x,y): y\in\{1,\ldots,n-1\}, x\in\{y+1,\ldots,n\} \} \borisspace .
\]
These points model the top/bottom view on a convex $2n$-gon, where the
point $(x,y)$ represents a triangle whose vertices are some point on
the bottom and the $x$th and $y$th points on the top, where the top
vertices are labelled $1,\ldots,n$ from left to right.

\begin{figure}
   \centering{
      \includegraphics{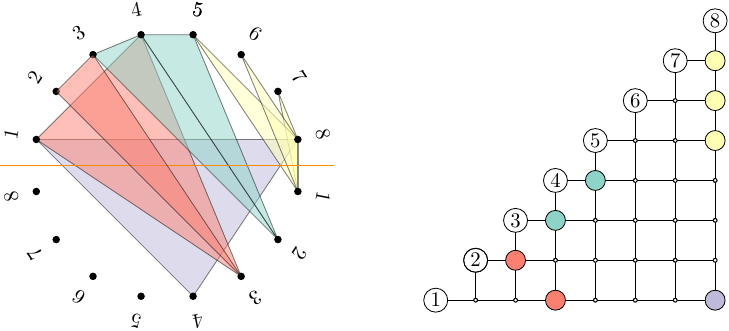}
   }
   \caption{The dot puzzle view of the top/bottom view. In this example,
     four rounds of the Dot Puzzle have been played.}
   \figlabel{point-view}
\end{figure}

The dot puzzle proceeds in $n$ rounds and during the $i$th round, the
player selects a set $Q_i\subseteq Q$ subject to certain constraints
that depend on the points selected in rounds $1,\ldots,i-1$.  In the
top/bottom view, the $i$th round determines which pairs of top vertices
form a triangle with the $i$th bottom vertex, where the bottom
vertices are labelled $1,\ldots,n$ from right to left.  

Of course, the constraints on which points can be selected during
round $i$ depend on the set of forbidden configurations and the set
$\bigcup_{j=1}^{i-1} Q_j$ of points played during previous rounds.
By proving bounds on $\sum_{i=1}^n |Q_i|$ we obtain bounds on the maximum
number of triangles obtained in the top-bottom view on a set of $2n$
points, i.e., bounds on $\ex'(2n, X)$.

\Figref{forbidden-colour}(a) shows restrictions on the locations of points
placed during a single round.  It is interpreted as follows:  If the
central point, $p=(x,y)$, is placed during round $i$, and we wish to
avoid some particular configuration, $c$, then we should not place any
points in the parts of the figure that have label $c$.  For instance,
if we wish to avoid the $c=\taco$ configuration, then we
should not place any points in the same row or column as $p$; such a point
creates a $\taco$ configuration in which the shared edge
joins a bottom vertex to a left (same row) or right (same column) vertex.

\Figref{forbidden-colour}(b) shows the restrictions on the locations
of points placed in subsequent rounds.  Its interpretation is similar
\figref{forbidden-colour}(a). For example, if we wish to avoid a
$\nested$ configuration and we place the central point, $p$, during round
$i$, then, in every round $j>i$, we should not place any point directly
to the left or directly below $p$.  Any such point creates 
a $\nested$ configuration in which the shared vertex is the left vertex
(to the left of $p$) or the right vertex (below $p$) of both triangles.

\begin{figure}
   \centering{
      \newlength{\ka}
      \setlength{\ka}{\textwidth}
      \begin{tabular}{c@{\hspace{1cm}}c}
        \includegraphics{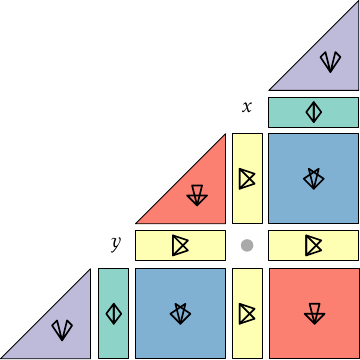} & 
        \includegraphics{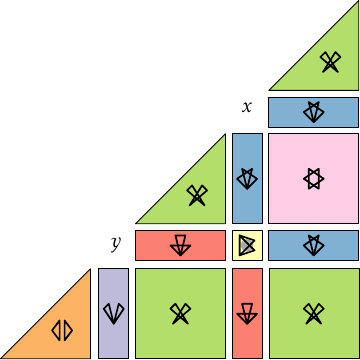} \\[1em]
        \includegraphics{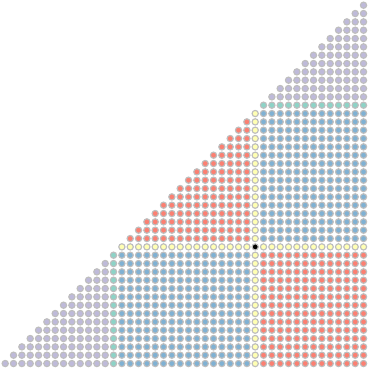} &
        \includegraphics{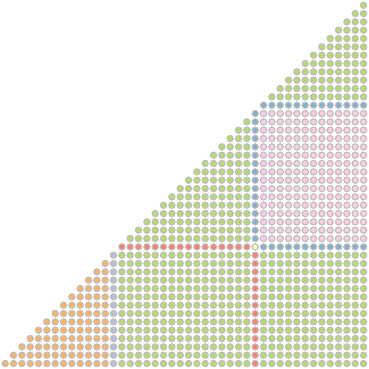} \\
        (a) & (b)
      \end{tabular}
   }
   \caption{The regions killed by forbidden configurations during (a)~the current round and (b)~subsequent rounds.}
   \figlabel{forbidden-colour}
\end{figure}

For any
$X\subseteq\{\taco,\mariposa,\bat,\nested,\crossing,\ears,\swords,\david\}$
and any $S\subseteq Q$, we define $\killed(X,S)$ as the subset of $Q$ that
can no longer be played in the dot puzzle game (for configurations in
$X$) if the points in $S$ have been played in previous rounds (these are
the points of $Q$ $\kappa$illed by $S$).  We use the complementary notation
$\survivors(X,S)=Q\setminus\killed(X,S)$ to be the subset of points in
$Q$ that can still be played in the dot puzzle game if the points is $S$
have been played in previous rounds.

\subsection{Some Warm-Up Exercises}

For the remainder of the paper, we will study $\ex'$ using the
dot puzzle view. Thus, all of our results are bounds on solutions to
these dot puzzles.

We say that a point set is \emph{non-decreasing} (respectively,
non-increasing) if, when sorted lexicographically, the $y$-coordinates of
the points form a non-decreasing (respectively, non-increasing) sequence.
A point set is \emph{increasing} (respectively, decreasing) if it is
non-decreasing (respectively, non-increasing) and no two of its points
points have the same $x$-coordinate or the same $y$-coordinate.

From \figref{forbidden-colour}, some previous upper bounds naturally fall
out.  Consider Bra\ss's results \cite{brass:turan} that $\ex(n,\{\nested\})\in
O(n^2)$.  For the game defined by $X=\{\nested\}$, we have the rules:\\
\centerline{\includegraphics[height=3cm]{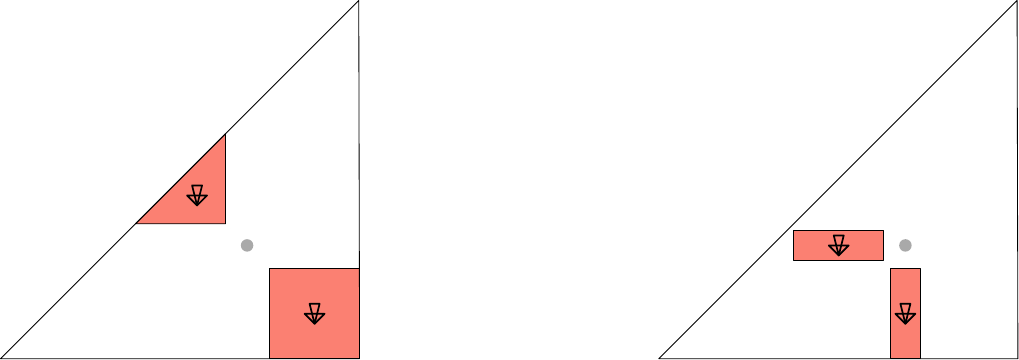}}
In particular, these rules imply that points selected during a
single round of the dot puzzle must be non-decreasing, and thus
at most $2n-3$ points can be selected take part in $Q_i$.  Thus
$\sum_{j=1}^{n}|Q_i| \le 2n^2-3n$, so $\ex'(n,\{\nested\})\in O(n^2)$
and the bound $\ex(n,\{\nested\})\in O(n^2)$ immediately follows from
\lemref{top-bottom}.

Similarly, we can almost recover the result $\ex(n,\{\ears, \swords,
\bat,\nested\})\le n$ of Bra\ss, Rote and
Swanepoel \cite{brass.rote.ea:triangles}.  Here, the rules are:\\
\centerline{\includegraphics[height=3cm]{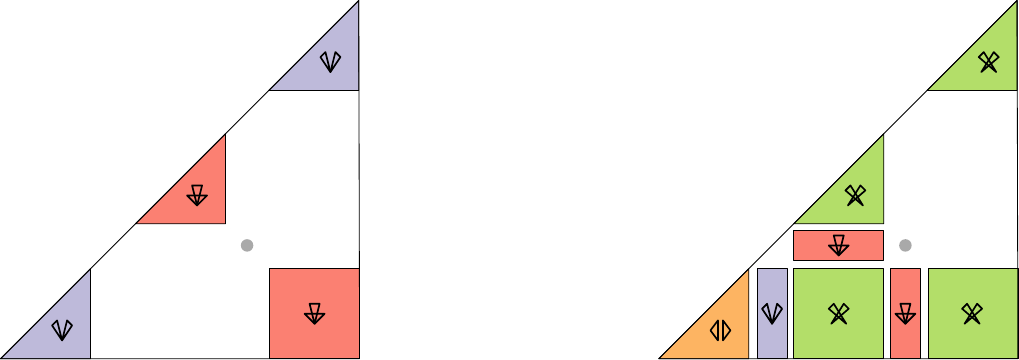}}
The rule for $\nested$ ensures that the set of points taken during a
single round form a non-decreasing point set.  The rules for points
allowed in subsequent rounds ensure that, after round $i$ any points
chosen are not below or to the left of the topmost-rightmost point
in $Q_i$.  Taken together, these rules imply that
\[
    \ex'(n,\{\ears,\swords,\bat,\nested\})\le \sum_{i=1}^n|Q_i| \le 3n-4 \borisspace ,
\]
since the union of $Q_i$ is a non-decreasing point set (whose size is
therefore at most $2n-3$), and each $Q_i$ shares at most one point with
$Q_{i+1}$.  The bound $\ex(n,\{\ears, \swords, \bat,\nested\})\in
O(n\log n)$ then follows from \lemref{top-bottom}.

\section{Results Based on Dot Puzzles}
\seclabel{new-results}

After this warm-up, and armed with the dot puzzle view, we are ready to prove
some new results.  We begin by proving several linear upper bounds.
From this point on, each proof of a theorem will begin with a picture,
similar to \figref{forbidden-colour}, that shows the rules of the dot puzzle
considered by the theorem.  For a point $q\in Q$, we use the notations
$\x(q)$ and $\y(q)$ to denote the x- and y-coordinates of $q$.

\subsection{Linear Upper Bounds}

\begin{thm}\thmlabel{taco-nested-crossing}
  $\ex'(n,\{\taco,\nested,\crossing\}) \in O(n)$.
\end{thm}

\begin{proof}
  \centerline{\includegraphics[height=3cm]{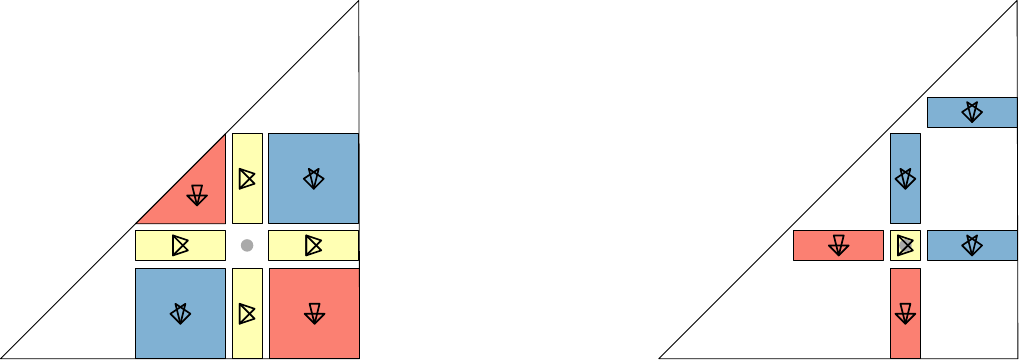}}
  Taking the union of the rules for $\taco$, $\nested$, and $\crossing$,
  we obtain the rule which ensures that during subsequent rounds we can
  not take a point from any column or row used in a previous round.
  The rules for $\taco$ ensure that the set of points taken in each $Q_i$
  includes at most one point in each row (or column).
  Therefore each new point played
  can be charged to a unique row, so the total number of points played
  is at most $n$.
\end{proof}

\begin{thm}\thmlabel{nested-crossing-ears}
  $\ex'(n,\{\nested,\crossing,\ears\})\in O(n)$.
\end{thm}

\begin{proof}
  \centerline{\includegraphics[height=3cm]{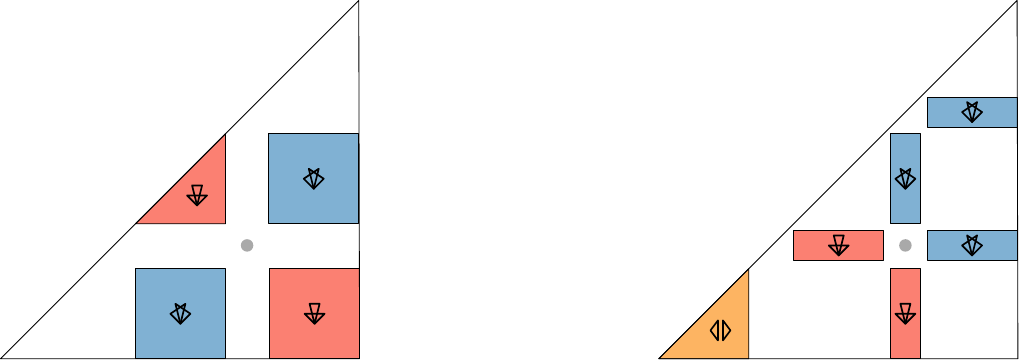}}
  Observe that if $Q_i$ contains $k>1$ points, $p_1,\ldots,p_k$ in a single column
  (or row), then each of the $k$ rows (or columns) containing one of these
  points is completely covered by $\killed(\{\nested,\crossing\},\{p_1,\ldots,p_k\})$, i.e., $Q_{i+1},\ldots,Q_n$ can not contain any points in
  these rows (or columns) (see \figref{nested-crossing-ears-i}).  
  Therefore, when summing $\sum_{i=1}^n |Q_i|$,
  the contribution of points that are not alone in their row or column
  is at most $2n$.  We therefore assume that each $Q_i$ contains at most
  one point from each row and column.  Since the rules for $\nested$ imply
  that $Q_i$ is non-decreasing, this implies that each $Q_i$
  is an increasing set of points.

  \begin{figure}[hbtp]
    \centering{
       \includegraphics{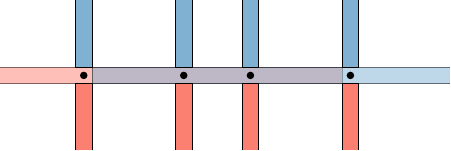}
    }
    \caption{A step in the proof of \thmref{nested-crossing-ears}.}
    \figlabel{nested-crossing-ears-i}
  \end{figure}

  Let $S=\bigcup_{i=1}^n Q_i$ and notice that $S$ contains at most
  one point from each row: each $Q_i$ contains at most one point
  in each row and the first time a point $p$ appears in some row,
  $\killed(\{\nested,\crossing\},\{p\})$ eliminates every other point
  from that row.  Therefore $|S|\le n$.  All that remains is to account for
  multiplicity; a single point in $S$ can appear in more than one $Q_i$.

  Now, because of the rules for $\crossing$, the condition that $Q_i$ is
  increasing is quite restrictive. In particular, if we consider the last
  (top rightmost) point, $p$, of $Q_i$, then it must be placed so that
  $\killed(\{\ears\}, \{p\})$ contains all of $Q_i$ except $p$ and the
  second-to-last point in $Q_i$ (see \figref{nested-crossing-ears}). That is,
  $Q_i$ contains $|Q_i|-2$ points that cannot appear in $Q_{i+1},\ldots,Q_n$. We can think of $Q_i$ as eliminating $|Q_i|-2$ points from $S$, so
  $\sum_{i=1}^n (|Q_i|-2) \le |S| \le n$,
  so $\sum_{i=1}^n |Q_i| \le 3n$.
  \begin{figure}[hbtp]
    \centering{
       \includegraphics{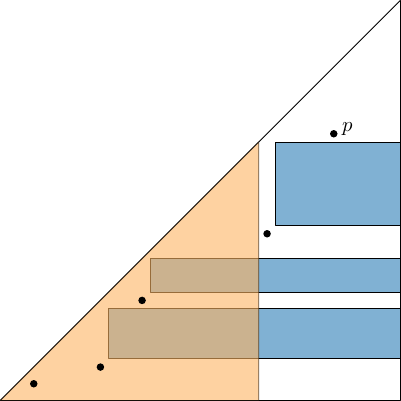}
    }
    \caption{Another step in the proof of \thmref{nested-crossing-ears}.}
    \figlabel{nested-crossing-ears}
  \end{figure}
\end{proof}

Our next four upper bounds depend on a simple lemma about forbidden
configurations of points.  We say that three points $a=(x_0,y_0)$,
$b=(x_0,y_1)$, and $c=(x_1,y_1)$ form a \emph{$\Gamma$-configuration} if
$y_0<y_1$ and $x_0<x_1$.

\begin{lem}\lemlabel{forbidden}
   Let $S$ be a subset of $\{1,\ldots,n\}^2$ with no three points $a$,
   $b$, and $c$ that form a $\Gamma$-configuration. Then $|S|\le 2n$.
\end{lem}

\begin{proof}
   If we remove the rightmost point from each row of $S$, then each
   column in what remains of $S$ contains at most one point. Otherwise, we
   could take $a$ to be the lowest point in a column, $b$ to be the highest
   point in the same column, and $c$ to be the removed rightmost point
   in $b$'s row.
\end{proof}

\begin{thm}\thmlabel{taco-nested-david}
  $\ex'(n,\{\taco,\nested,\david\})\in O(n)$.
\end{thm}

\begin{proof}
  \centerline{\includegraphics[height=3cm]{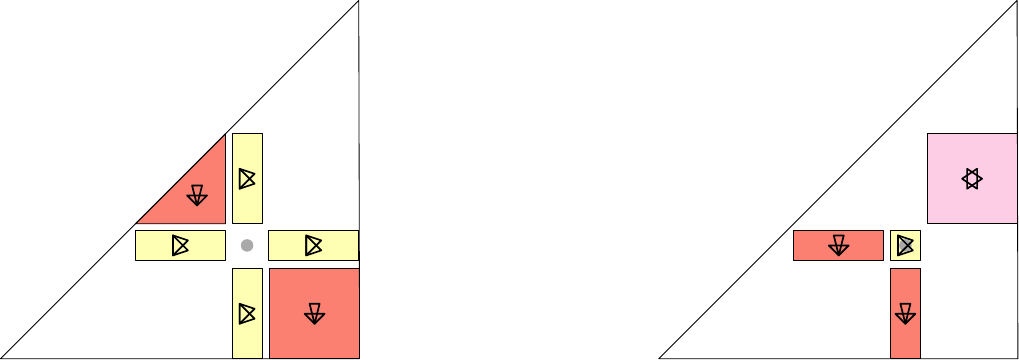}}
  Let $S=\bigcup_{i=1}^n Q_i$ be the set of points played in a solution
  to the resulting dot puzzle.  Note that, by the inclusion of $\taco$,
  each element of $S$ appears in exactly one $Q_i$, so $|S|=\sum_{i=1}^n
  |Q_i|$ is the quantity we are interested in bounding.

  Next, we claim that $S$ does not contain any three points $a$, $b$,
  and $c$ forming a $\Gamma$-configuration. Refer to \figref{forbidden}.
  Suppose, for the sake of contradiction, that this were not the
  case and that $a\in Q_i$, $b\in Q_j$, and $c\in Q_k$.  We have
  $a\in\killed(\{\nested\},\{b\})$ so $i\le j$ and the rules for $\taco$
  implies $i\neq j$, so $i<j$.  We also have $b\in\killed(\{\nested\},
  c)$ so $j\le k$.  Finally, we have $c\in\killed(\{\david\}, a)$, so
  $k \le i$.  Taken together, this gives the contradiction $i < j \le
  k \le i$.  Therefore, $S$ contains no $\Gamma$-configuration and applying \lemref{forbidden} then implies that $|S|\le 2n$.
  \begin{figure}[hbtp]
    \centering{
      \includegraphics{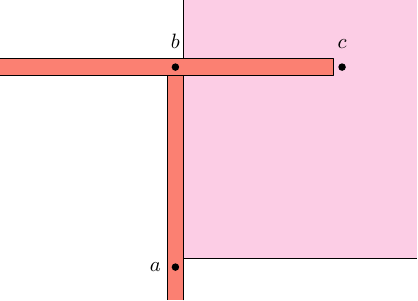}
    }
    \caption{The proof of \thmref{taco-nested-david}.}
    \figlabel{forbidden}
  \end{figure}
\end{proof}

\begin{thm}\thmlabel{crossing-swords}
  $\ex'(n,\{\crossing,\swords\}) \in O(n)$.
\end{thm}

\begin{proof}
  \centerline{\includegraphics[height=3cm]{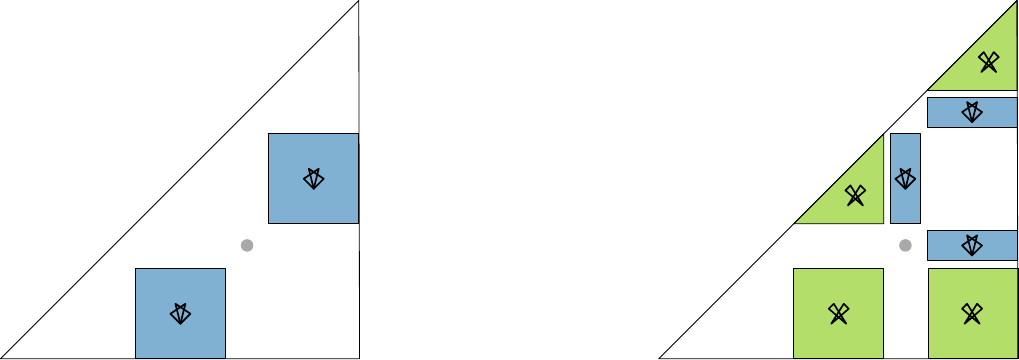}}
  Let $S=\bigcup_{i=1}^n Q_i$. We claim that $S$ has no $\Gamma$-configuration so, by \lemref{forbidden}, $|S|\le 2n$.  To see this, assume $S$
  contains a $\Gamma$-configuration  $a$, $b$, and $c$. Then the rules
  for $\crossing$ imply that no set $Q_i$ contains both $a$ and $c$.
  However, the rules for $\swords$ and $\crossing$ imply that, if $a\in
  Q_i$, $b\in Q_j$ and $c\in Q_k$ then 
  \[  i \ge j \ge k \ge i \borisspace . \] 
  (See \figref{crossing-swords}).  But this is a contradiction,
  since it implies that $i=j=k$.

  \begin{figure}[htbp]
    \centering{
      \includegraphics{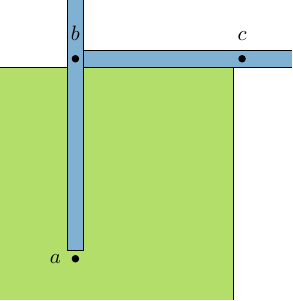} 
    }
    \caption{The proof of \thmref{crossing-swords}.}
    \figlabel{crossing-swords}
  \end{figure}

  Now, for some $Q_i$, consider a point $p\in Q_i$ with minimum x-coordinate
  and, in case more than one such point exists, take the the one that
  minimizes $\y(p)$.  Observe that $\killed(\{\crossing,\swords\},\{p\})
  \supseteq Q_i\setminus\{p\}$.  Indeed, every point directly above
  and every point to the right of $p$ is killed by $p$ or cannot be
  included in $Q_i$ because of the rule for $\crossing$.  Therefore,
  $Q_i$ eliminates at least $|Q_i|-1$ points of $S$. It follows that
  $\sum_{i=1}^n |Q_i| \le |S|+n \le 3n$.
\end{proof}

\begin{thm}\thmlabel{nested-ears-david}
  $\ex'(n,\{\nested,\ears,\david\}) \in O(n)$.
\end{thm}

\begin{proof}
   \centerline{\includegraphics[height=3cm]{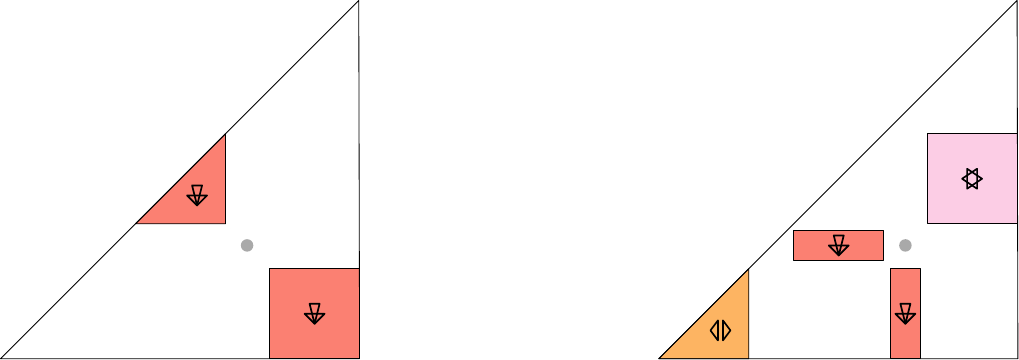}}
   Let $S=\bigcup_{i=1}^n Q_i$.  We will first show that, for each
   $i\in\{1,\ldots,n\}$, there are at most two points of $Q_i$
   that appear in $Q_{i+1},\ldots,Q_n$.  Refer to \figref{nested-ears-david}. 
   Define
   \[
        Q_i^* = \{ (x,y)\in Q_i : x\ge\max\{y(p):p\in Q_i\} \} \borisspace ,
   \]
   and let $p$ and $r$ be the top rightmost and bottom leftmost
   points in $Q_i^*$, respectively.  If there is more than one
   point in $Q_i^*$ with $x$-coordinate equal to $\x(r)$ (as in
   \figref{nested-ears-david}(a)) then we define $q$ to be the
   highest such point (note that this includes the case where $p=r$).
   Otherwise (as in \figref{nested-ears-david}(b)), we define $q$ to
   be the rightmost point with $y$-coordinate $\y(r)$ (note that this
   includes the case where $q=r$).  Now, observe that
   $\killed(\{\nested,\ears,\david\}, \{p,q,r\}) \supseteq Q_i\setminus
   \{p,q\}$, so $p$ and $q$ are the only points of $Q_i$ that can appear
   again in $Q_{i+1},\ldots,Q_n$. (Note that \figref{nested-ears-david} only illustrates the case in which $y(p)=x(r)$; if $y(p) < x(r)$, then even $p$
   is contained in $\killed(\{\nested,\ears,\david\}, \{p,q,r\}$.)
   \begin{figure}
     \centering{
       \begin{tabular}{cc}
         \includegraphics{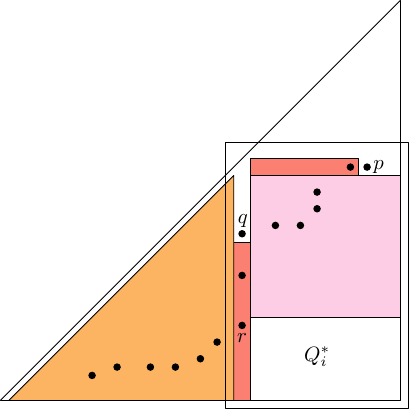} &
         \includegraphics{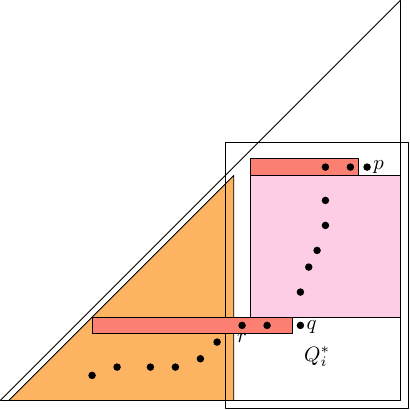} \\
         (a) & (b)
       \end{tabular}
     }
     \caption{A step in the proof of \thmref{nested-ears-david}.}
     \figlabel{nested-ears-david}
   \end{figure}

   We can therefore think of $Q_i$ as eliminating $|Q_i|-2$ points
   from $S$, so $\sum_{i=1}^n(|Q_i|-2) \le |S|$, which implies that
   $\sum_{i=1}^n |Q_i| \le |S|+2n$.   All that remains now is to bound
   $|S|$.

   For each $i\in \{1,\ldots,n\}$, let $Q_i'$ be obtained from $Q_i$ by
   removing the leftmost point in each row.  Let $S'=\bigcup_{i=1}^n
   Q_i'$.  We claim that $S'$ contains no $\Gamma$-configuration so,
   by 
   \lemref{forbidden}, $|S'|\le 2n$.  To see why this is so,
   suppose that $S'$ contains a $\Gamma$-configuration $a$, $b$, and $c$.
   Then, as argued in the proof of \thmref{taco-nested-david}, it must be
   that $a,b,c\in Q_i'$ for some $i$.  However, this contradicts the fact
   (due to $\nested$) that $Q_i$ is non-decreasing since the leftmost
   point, $b'$, of $Q_i$ in the same row as $b$ is to the left of $a$
   (see \figref{nested-ears-david-ii}).

   \begin{figure}
     \centering{
       \includegraphics{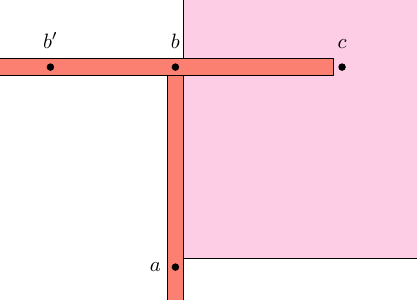} 
     }
     \caption{Another step in the proof of \thmref{nested-ears-david}.}
     \figlabel{nested-ears-david-ii}
   \end{figure}

   Therefore, $|S'|\le 2n$.   Now let $S''=S\setminus S'$.  We
   claim that $S''$ also satisfies the conditions of \lemref{forbidden}.
   Indeed, by the same reasoning as above, if there were $a,b,c\in
   S''$ forming a $\Gamma$-configuration, then it must
   be that $a,b,c\in Q_i\setminus Q_i'$ for some $i$.  But this is a
   contradiction since $b$ and $c$ are in the same row, and $Q_i\setminus
   Q_i'$ contains at most one point per row.

   Wrapping up, we have $|S|=|S'|+|S''| \le 4n$ so $\sum_{i=1}^n|Q_i| \le 6n$.
\end{proof}

\begin{thm}\thmlabel{nested-bat-david}
  $\ex'(n,\{\nested,\bat,\david\})\in O(n)$.
\end{thm}

\begin{center}
   \includegraphics[height=3cm]{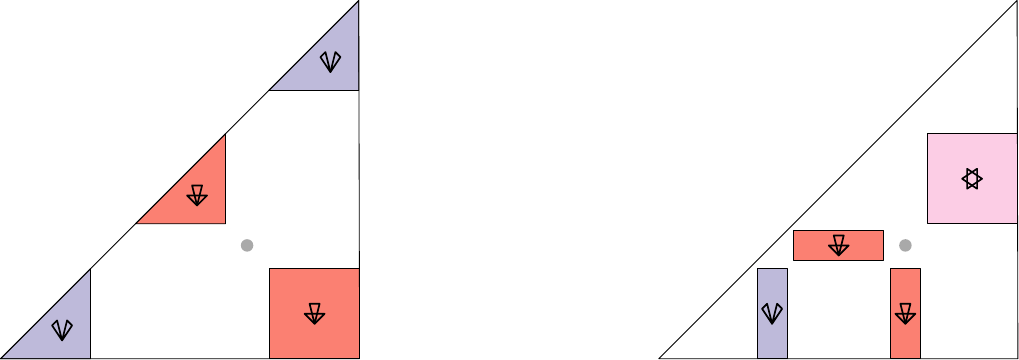}
\end{center}

\begin{proof}
  Consider the set $Q_i$ played during some round $i$.  The rules for
  $\nested$ imply that $Q_i$ is non-decreasing.  Consider the following
  subsets of $Q_i$:
  \begin{enumerate}
    \item the set $L_i$ of points in the leftmost column of $Q_i$;
    \item the set $T_i$ of points in the topmost row of $Q_i$; 
    \item the set $B_i$ of points in the bottommost row of $Q_i$; and 
    \item the set $N_i=Q_i\setminus(L_i\cup T_i\cup B_i)$.
  \end{enumerate}
  Let $p_i$ denote the lowest leftmost point of $Q_i$ (the unique point
  in $B_i\cap L_i$).  The rule for $\bat$ implies that every point of
  $N_i$ is contained in $\killed(\{\david\},\{p_i\})$.
  
  Refer to \figref{nested-bat-david}.  Observe that, for any point $p\in
  N_i$, the entire row containing $p$ is killed in the sense that it is
  contained in $\killed(\{\david\},\{p_i\})\cup\killed(\{\nested\},\{p\})$.
  Next, let $t_i$ be
  the topmost point in $L_i$ and observe that, for any point $p\in
  L_i\setminus\{t_i,p_i\}$, the entire row containing $p$ is killed by
  $\killed(\{\david\},\{p_i\})\cup\killed(\{\nested\},\{t_i,p\})$.

  Consider the operation of removing the leftmost point from each
  row of $Q_i$, for each $i\in \{1,\ldots,n\}$.  We claim that
  this removes a total of at most $4n$ points from $Q_1,\ldots,Q_n$.
  Indeed, by the preceding discussion, if we remove a point $p\in N_i$
  or $p\in L_i\setminus\{t_i,p_i\}$ then this point can be charged to
  a row that is never used again in $Q_{i+1},\ldots,Q_n$.  For each
  round $i\in\{1,\ldots,n\}$, there are only three other choices for $p$:
  $p=p_i$, $p=t_i$, or $p$ is the leftmost point in $T_i$.  Thus, we
  can charge each row for removing at most one point and each round
  for removing at most 3 points.

  \begin{figure}
    \begin{center}\includegraphics{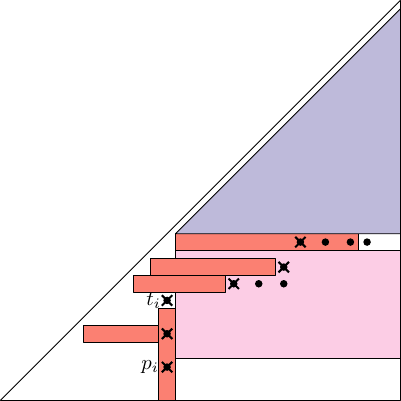}\end{center}
    \caption{The proof of \thmref{nested-bat-david}.} 
    \figlabel{nested-bat-david}
  \end{figure}

  For each $i\in\{1,\ldots,n\}$, let $Q_i'$ be the subset of $Q_i$
  obtained by removing the leftmost point in each row and observe that
  $Q_i'$ is an increasing set of points.  By the preceding discussions
  $\sum_{i=1}^n|Q_i| \le 4n+\sum_{i=1}^n|Q_i'|$.  Let $S=\bigcup_{i=1}^n
  Q_i'$.  We claim that $S$ contains no $\Gamma$-configuration.
  To see why this is so, observe that, since each $Q_i'$ is an
  increasing set, if two points $a\in Q_i'$ and $b\in Q_j'$ of $S$
  are in the same column then $i\neq j$.  Assume $b$ is above $a$,
  then $a\in\killed(\{\nested\},\{b\})$ so $i < j$.  Now, if $c\in
  Q_k'$ is to the right of $b$, then $b\in\killed(\{\nested\},\{c\})$,
  so $j \le k$. Therefore, $i < j \le k$, but this is not possible
  since $c\in\killed(\{\david\},\{a\})$, so $k \le i$. Therefore, by
  \lemref{forbidden}, $|S|\le 2n$.

  All that remains is to account for points in $S$ that are played
  multiple times.  In each $Q_i$ there are at most two points that can
  be played in subsequent rounds: the
  rightmost point in $B_i$ and the rightmost point in $T_i$.  We charge
  each occurrence of such repeated points to the rounds in which they
  are these extreme points. In this way, each round is charged for at
  most two such points and the total contribution of these points to
  $\sum_{i=1}^n |Q_i|$ is at most $2n$.

  In summary,
  \[
    \sum_{i=1}^n |Q_i| \le 
    4n+\sum_{i=1}^n |Q_i'| \le  
    4n+2n + |S| \le 8n \borisspace . \qedhere
  \]
\end{proof}

\subsection{Forbidding Swords}

Next we focus on the configuration $\swords$ and give linear upper bounds
bounds on $\ex'(n,\{\swords, \taco\})$ and $\ex'(n,\{\swords, \nested\})$.
We begin with another lemma about forbidden configurations of points
that is similar to \lemref{forbidden}.  We say that a point $(x_i,y_i)$
\emph{SE-dominates} a point $(x_j,y_j)$ if $x_i > x_j$ and $y_i < y_j$.
We say that three points $a=(x_0,y_0)$, $b=(x_0,y_1)$, and $c$
form an \emph{obtuse-$L$-configuration} if $y_1<y_0$ and $c$ SE-dominates
$b$.

\begin{figure}
  \centering{\includegraphics{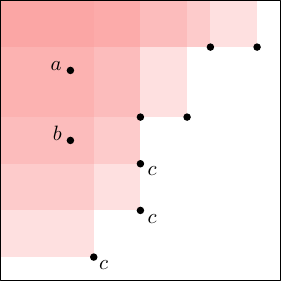}}
  \caption{The proof of \lemref{forbidden-ii}.}
  \figlabel{se-domination}
\end{figure}

\begin{lem}\lemlabel{forbidden-ii}
   Let $S$ be a subset of $\{1,\ldots,n\}^2$ with no three points $a$,
   $b$, and $c$ forming an obtuse-$L$-configuration.  Then $|S|\le 3n$.
\end{lem}

\begin{proof}
  Refer to \figref{se-domination}.  Consider the \emph{Pareto boundary}
  $P\subseteq S$ containing each point of $S$ that is not SE-dominated
  by any other point in $S$.  The set $P$ is non-decreasing, so it
  has size at most $2n$.  We claim that the set $S\setminus P$ has at
  most one point in each column, so $|S|\le 3n$.  To see why this claim
  is true, observe that if some column of $P\setminus S$ contains two
  points $a$ and $b$ with $a$ above $b$, then at least one point $c$
  in $P$ SE-dominates $b$, so that $a$, $b$, and $c$ would form the
  obtuse-$L$-configuration.
\end{proof}

\begin{thm}\thmlabel{taco-swords}
  $\ex'(n,\{\taco,\swords\}) \in O(n)$.
\end{thm}

\begin{proof}
  \centerline{\includegraphics[height=3cm]{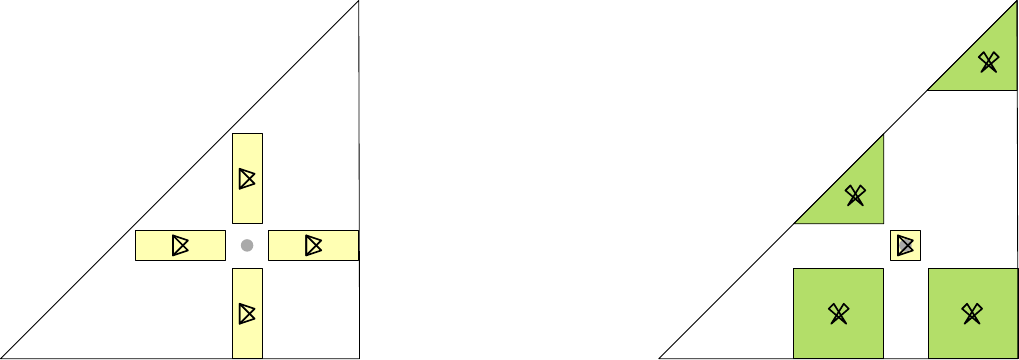}} Let
  $S=\bigcup_{i=1}^n Q_i$. Then the rules for $\taco$ imply that
  $|S|=\sum_{i=1}^n |Q_i|$, so it suffices to bound~$|S|$.  We claim that
  $S$ contains no obtuse-$L$-configuration so, by \lemref{forbidden-ii},
  $|S|\le 3n$.

  Suppose there were $a\in Q_i$, $b\in Q_j$, and $c\in Q_k$ forming an
  obtuse-$L$-configuration.  Now, $a\in\killed(\{\swords\}, \{c\})$ and
  $c\in\killed(\{\swords\}, \{a\})$, so it must be that $i=k$.  The same
  argument, applied to $b$ and $c$, implies that $j=k$, so $i=j=k$.
  But this is a contradiction since the rules for $\taco$ imply that
  $i\neq j$.
\end{proof}

\begin{thm}\thmlabel{nested-swords}
  $\ex'(n,\{\nested,\swords\}) \in O(n)$.
\end{thm}

\begin{proof}
   \centerline{\includegraphics[height=3cm]{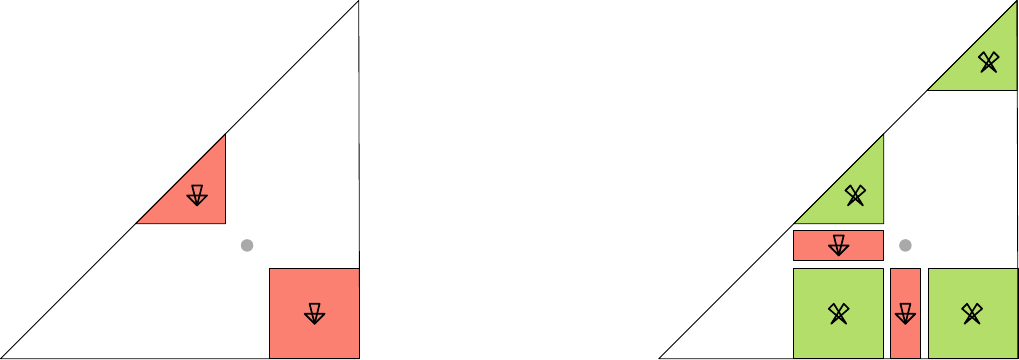}}
   Let $S=\bigcup_{i=1}^n Q_i$.  We claim that $S$ is non-decreasing.
   Indeed, the rule for $\nested$  prevents two decreasing points from
   being played in the same round, while the rule for $\swords$ prevents
   two decreasing points from being played in different rounds.
   This implies that $|S|\le 2n$.  What remains is to account for points 
   of $S$ that are played in multiple rounds.

   Consider the graph $G$ with vertex set $S$ that contains an edge
   $uw$ if and only if the $x$-coordinate of $u$ is equal to the
   $y$-coordinate of $w$.  We claim that $G$ is 4-colourable.  To prove
   this, we partition $S$ into two sets $A$ and $B$ and show that each
   of the graphs $G[A]$ and $G[B]$ induced by $A$ and $B$ is 2-colourable
   (in fact, $G[A]$ and $G[B]$ are each forests). Thus, if we colour $A$
   with colours $\{1,2\}$ and $B$ with colours $\{3,4\}$, then we obtain
   a 4-colouring of $G$.

   Refer to \figref{four-colouring}.  Remove the bottom-most point of
   $S$ from each column and what remains is the set $A$.  Observe that,
   since $S$ is non-decreasing, $A$ contains at most one point per row.
   Imagine directing the edges of $G[A]$ from right to left (top to
   bottom).  This directed graph is obviously acyclic and, since each row contains at most one point of $A$, has maximum in-degree 1.  Therefore $G[A]$
   is a forest and can be 2-coloured using the colours $\{1,2\}$. 
   \begin{figure}
      \centering{
       \begin{tabular}{ccc}
         \includegraphics{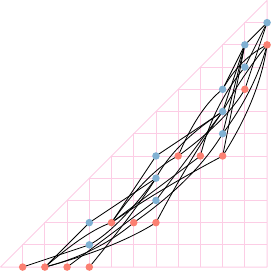} &
         \includegraphics{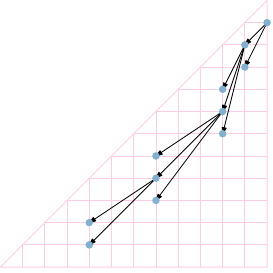} &
         \includegraphics{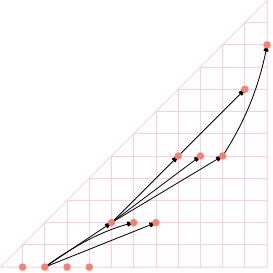} \\
         $S$ & $G[A]$ & $G[B]$
      \end{tabular}
      }
      \caption{Four-colouring the graph $G$ in the proof of 
               \thmref{nested-swords}.}
      \figlabel{four-colouring}
   \end{figure}
   By a similar argument, using the fact that each column contains at
   most one point of $B$, the graph $G[B] = G[S\setminus A]$ can be
   2-coloured using the colours $\{3,4\}$.

   The resulting 4-colouring of $G$ partitions $S$ into 4 colour classes
   $S_1,\ldots,S_4$.  We now argue that each of these colour
   classes contributes $O(n)$ to $\sum_{i=1}^n|Q_i|$.
   Consider a (new) directed graph $H_j=(S_j,E_j)$ that
   contains the edge $\overrightarrow{uw}$ if and only if
   $u$ kills $w$, i.e., $w\in\killed(\{\nested,\swords\},\{u\})$.
   We claim that this graph is \emph{complete}, i.e., for any $u,w\in
   S_j$ at least one of $\overrightarrow{uw}$ or $\overrightarrow{wu}$ is
   in $E_j$.  To see why this is so, consider any two distinct points $u,w\in S_j$
   with $u=(x_0,y_0)$ and $w=(x_1,y_1)$.  Since $S$ (and hence $S_j$)
   is non-decreasing, we may assume without loss of generality that
   $x_0\le x_1$ and $y_0\le y_1$.
   There are five cases to consider: 
   \begin{enumerate}
    \item $y_0=y_1$, so $x_0 < x_1$. In this case, $u\in\killed(\{\nested\},w)$, so $\overrightarrow{wu}\in E$.
    \item $x_0=x_1$, so $y_0 < y_1$. In this case, $u\in\killed(\{\nested\},w)$, so $\overrightarrow{wu}\in E$.
    \item $x_0 < x_1$, $y_0 < y_1$, and $y_1 > x_0$.  In this case, $w\in\killed(\{\swords\}, \{u\})$, so $\overrightarrow{uw}\in E$.
    \item $x_0 < x_1$, $y_0 < y_1$, and $y_1 < x_0$.  In this case, $u\in\killed(\{\swords\}, \{w\})$, so $\overrightarrow{wu}\in E$.
    \item $x_0 < x_1$, $y_0 < y_1$, and $y_1 = x_0$.  This case cannot occur
      since, in this case, the graph $G$ contains the edge $uw$, so $u$ and $w$
      are assigned different colours and at most one of them is $j$.
  \end{enumerate}

  Suppose now that $H_j$ contains a directed cycle
  $C=u_0,\ldots,u_{\ell-1}$.  Since the points in
  this cycle all kill each other, i.e., $u_{k+1\bmod
  \ell}\in\killed(\{\nested,\swords\},\{u_k\})$), it must be
  the case that all the vertices of $C$ are played in the same
  round $i'\in\{1,\ldots,n\}$ and never played again, i.e.,
  $V(C)\subseteq Q_{i'}$ and $V(C)\cap Q_{j'}=\emptyset$ for all
  $j'\in\{1,\ldots,n\}\setminus \{i'\}$.

  Now, if we repeatedly find a cycle in $H_j$ and remove its vertices, we
  will eventually be left with an acyclic subgraph $H_j'$ with vertex set
  $S'_j\subseteq S_j$.  From the preceding discussion, we know that
  each cycle vertex we remove contributes only 1 to $\sum_{i=1}^n |Q_i|$:
  \[
        \sum_{i=1}^n |(S_j\setminus S_j')\cap Q_i| 
             = |S_j\setminus S_j'| \borisspace .
  \]

  Finally, we are left with the complete acyclic subgraph $H_j'$ with vertex
  set $S_j'$ and whose topological sort order we denote by $\prec$.  Now,
  if $Q_i$ contains vertices $v_1\prec\cdots\prec v_{k}$ 
  of $H_j'$, then $v_1$ kills all of $v_2,\ldots,v_{k}$ so that these
  vertices can not appear in any $Q_{i'}$ with $i'>i$.  This implies that
  \[
      \sum_{i=1}^n (|S_j'\cap Q_i|-1) \le |S_j'| \borisspace ,
  \]
  so $\sum_{i=1}^n |S_j'\cap Q_i| \le |S_j'|+n$.  Putting everything together, we have
  \begin{align*}
     \sum_{i=1}^n|Q_i| 
         & = \sum_{j=1}^4\left(\sum_{i=1}^n |Q_i\cap S_j|\right) \\
         & = \sum_{j=1}^4\left(\sum_{i=1}^n (|Q_i\cap(S_j\setminus S_j')|
              + |Q_i\cap S_j'|\right) \\
         & \le \sum_{j=1}^4\left(|S_j\setminus S_j'| + |S_j'|+n\right) \\
         & = |S|+4n \le 6n \borisspace . \qedhere 
  \end{align*}
\end{proof}

\subsection{Monotone Matrices, Tripod Packing, and 2-Comparable Sets}
\seclabel{tripods}

In this section, we discuss $\ex(n,\{\taco,\nested\})$:\\
\centerline{\includegraphics[height=3cm]{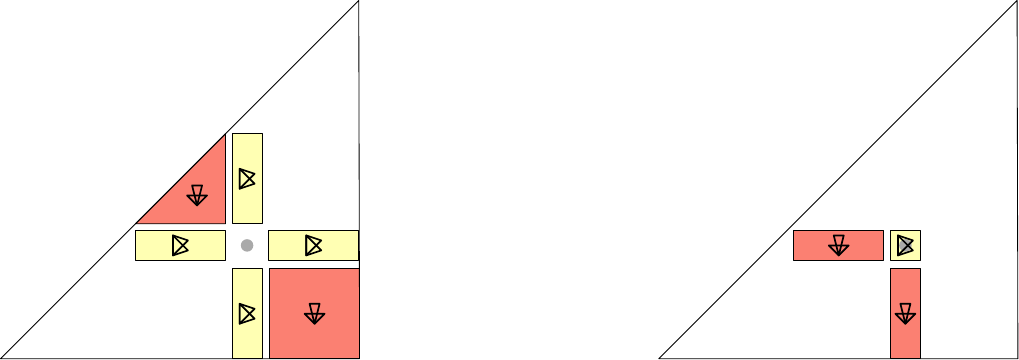}}
Determining the asymptotics of $\ex(n,\{\taco,\nested\})$ was given
explicitly as an open problem in the conclusion of Bra\ss's paper.
We spent more than a year working on this problem and this work included
computer searches for a variant of the problem played on the square grid
$\{1,\ldots,n\}^2$.\footnote{Playing on the square grid does not change
the asymptotics of the problem. Any solution for $\{1,\ldots,\lfloor n/2\rfloor-1\}^2$ can
be used as a solution for the triangular grid $Q$ and any upper bound for
$\{1,\ldots,n\}^2$ is also an upper bound for the triangular grid $Q$.}
Using the results of these computer searches in the Online Encyclopedia of
Integer Sequences \cite{oeis:a070214}, we discovered that this problem,
when played on the square grid, is equivalent to several other known
problems. See \figref{tripods}.

\begin{figure}
  \centering{
    \includegraphics{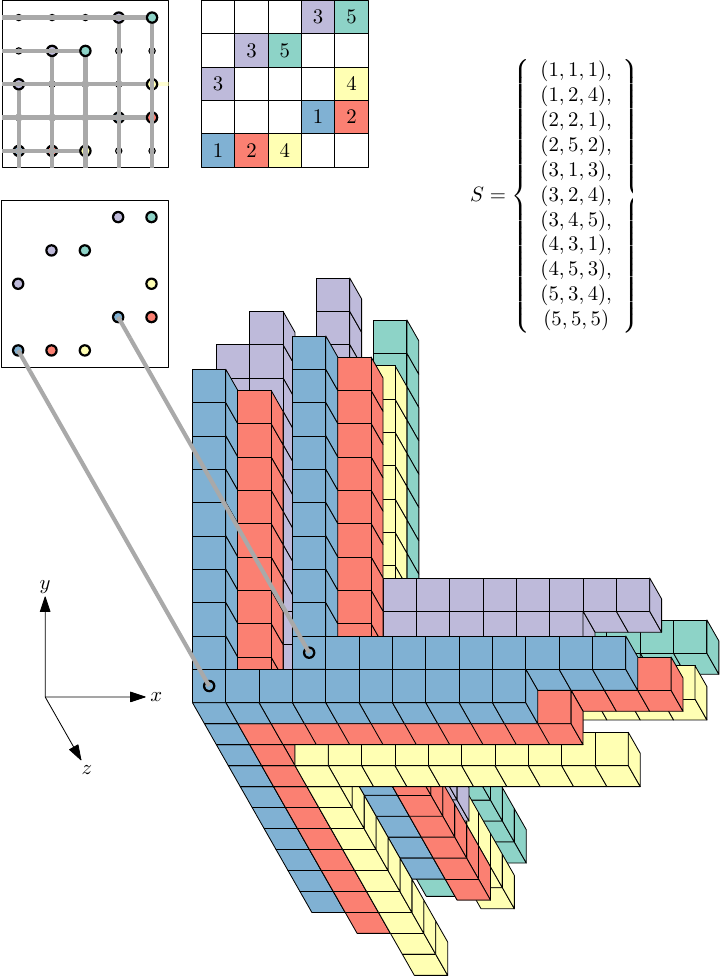}
  }
  \caption{The dot puzzle induced by excluding the taco and nested
       configurations has already been studied under several equivalent
       formulations.}
  \figlabel{tripods}
\end{figure}

\begin{enumerate}
  \item \emph{Monotone matrix problem}: How many values from
  $\{1,\ldots,n\}$ can one write in an $n\times n$ matrix, so that each
  row is increasing from left-to-right, each column is increasing from
  bottom-to-top, and for each $i\in\{1,\ldots,n\}$, the positions of $i$
  in the matrix form an increasing sequence?

  \item \emph{Tripod packing problem}: A \emph{tripod} with top $p=(x,y,z)\in\R^3$
  is the union of three closed rays originating at $p$ and directed in
  the positive x-, y-, and z-directions, i.e.,
  \[ \tripod(x,y,z) = \bigcup_{0\le t<\infty}\{ (x+t,y,z),(x,y+t,z),(x,y,z+t)\}
      \borisspace . \]
  How many disjoint tripods can be packed with tops
  in $\{1,\ldots,n\}^3$?

  \item \emph{2-comparable sets of triples problem}.  Two triples of
  integers $(a_1,a_2,a_3)$ and $(b_1,b_2,b_3)$ are \emph{2-comparable}
  if $a_i< b_i$ for at least two values of $i\in\{1,2,3\}$ or $a_i > b_i$
  for at least two values of $i\in\{1,2,3\}$.  What is the largest set,
  $S$, of pairwise 2-comparable triples one can make whose entries come
  from $\{1,\ldots,n\}$?
\end{enumerate}

Several simple and natural constructions give lower bounds
of $\Omega(n^{3/2})$ for these problems.  However,
this bound is not tight. A sequence of 
recursive constructions has steadily raised this lower bound
\cite{gowers.long:length,stein:combinatorial,stein:packing,stein.szabo:algebra,tiskin:packing}.
The current record is held by Gowers and Long \cite{gowers.long:length},
who describe a construction of size $\Omega(n^{1.546})$.

\begin{thm}[Gowers and Long \cite{gowers.long:length}]\thmlabel{taco-nested-lb}
  $\ex'(n,\{\taco,\nested\}) \in \Omega(n^{1.546})$.
\end{thm}

The only known upper bound for this problem comes from the fact that a
solution to this problem gives a solution to the Ruzsa-Szemer{\'e}di
induced matching problem \cite{ruzsa.szemeredi:triple}.  
\begin{enumerate}
  \setcounter{enumi}{3}
  \item \emph{Induced-matching problem}: What is the maximum number
  of edges in a bipartite graph $G=(A,B,E)$ with $|A|=|B|=n$ such that
  $E$ can be partitioned into $n$ \emph{induced matchings} $M_1,\ldots,M_n$?
  That is, each $M_i$ is a matching, and for any two edges $e,f\in M_i$
  there is no edge in $E$ that joins an endpoint of $e$
  to an endpoint of $f$.
\end{enumerate}
It is simple to verify that if one takes a 2-comparable set of triples
$S=\{(a_i,b_i,c_i):i\in\{1,\ldots,m\}\}$ then the bipartite graph $G=(A,B,E)$
with $A=B=\{1,\ldots,n\}$ and
  \[
     E = \{ (b_j,c_j): j\in \{1,\ldots,m\}\}
  \]
satisfies the conditions of the induced matching problem, with the
partition into matchings given by
\[
    M_i = \{ (b,c) : (i,b,c)\in S \} \borisspace . 
\]
Thus, any upper bound for the induced-matching problem is also an upper bound
on the size of a 2-comparable set of triples.

Known upper bounds for the induced matching problem are barely
subquadratic, with the current record being held by Fox's improved
version of the triangle removal lemma \cite{fox:new}, which gives an
upper bound of $n^{2}/e^{\Omega(\log^* n)}$.  See the discussion, for
example, in Gowers and Long \cite{gowers.long:length}.  Lower bounds for
the induced matching problem are surprisingly high; a result of Behrend
\cite{behrend:on} can be used to construct $n$ vertex graphs with 
$n^2/e^{O(\sqrt{\log n})}$ edges that can be decomposed into induced matchings.

\begin{thm}[Fox \cite{fox:new}]\thmlabel{taco-nested-ub}
  $\ex'(n,\{\taco,\nested\}) \in n^2/e^{\Omega(\log^* n)}$.
\end{thm}

While discovering these results, we noticed that the relationships
between some of these problems have gone unnoticed.  Here we make 
a few bibliographic notes.

\begin{itemize}
  \item  Bra\ss\ \cite{brass:turan} seems to have been unaware that
  the question he posed was equivalent to tripod packing and monotone
  matrices (or, like us, had never heard of these problems).
  
  \item Tiskin \cite{tiskin:packing}, apparently unaware of the
  relation between tripod packing and induced matchings, proved an upper
  bound of $o(n^2)$ for tripod packing. His proof does not depend on any
  properties of tripod packing that are not also true for induced matchings,
  and uses the same tools
  (namely Szemeredi's Regularity Lemma) as the original upper bounds for
  the induced matching problem.

  \item Gowers and Long \cite{gowers.long:length} seem to be unaware
  that the problem on 2-comparable sets of triples was studied under
  other names.
  
  \item Gowers and Long \cite{gowers.long:length} arrived at 2-comparable
  sets as a relaxation of a problem (the size of the largest 2-increasing
  sequence of triples)  proposed by Loh \cite{loh:directed}.  In his
  discussion of this problem, Loh formulates a restricted version of the
  induced matching problem, in which the matching must satisfy a certain
  $\Sigma$-free property and expresses hope \cite[remark on page~9]{loh:directed}
  that this restricted version has an $O(n^{3/2})$ upper bound.  However,
  solutions for tripod packing correspond to $\Sigma$-free induced
  matchings, so $\Sigma$-free induced matchings of size $\omega(n^{3/2})$
  are already known.
\end{itemize}

In our context, the only new observation we have pertains to
$\ex'(n,\{\taco,\nested,\bat,\ears\})$:\\
\centerline{\includegraphics[height=3cm]{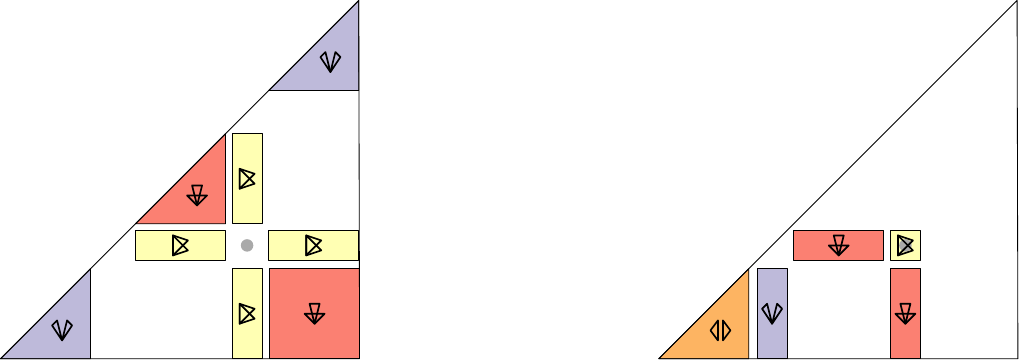}}
We observe that $\ex'(n,\{\taco,\nested\})\in
O(\ex'(n,\{\taco,\nested,\bat,\ears\})$ so these two functions
therefore have the same asymptotic growth.  This comes from the
fact that a solution for the dot puzzle of size $n$ resulting from
$X=\{\taco,\nested\}$ can be used as a solution for the dot puzzle of size
$2(n+1)$ resulting from $X=\{\taco,\nested,\bat,\ears\}$ by only playing
in the lower-right quadrant.  When played this way, the extra restrictions
caused by $\bat$ and $\ears$ do not affect the lower-right quadrant.

\begin{thm}
   $\ex'(n,\{\taco,\nested\}) \in \Theta(\ex'(n,\{\taco,\nested,\bat,\ears\})$.
\end{thm}

\subsection{Lower Bounds}
\seclabel{lower-bounds}

Finally, we finish up with some $\Omega(n^2)$ lower bound constructions.
In each case, a matching upper bound follows from one of the results in
Bra\ss\ \cite{brass:turan}.  The following are essentially ``proofs by figure'' in which a brief description of the solution is illustrated alongside the rules of each dot puzzle. In order to avoid floors and ceilings, we assume $n$ is even.

\begin{thm}\thmlabel{taco-david-crossing-bat-ears}\label{dilwad}
$\ex(n,\{\taco,\david,\crossing,\bat,\ears\})\in\Theta(n^2)$.
\end{thm}

\begin{proof}
\centerline{
   \includegraphics[height=3cm]{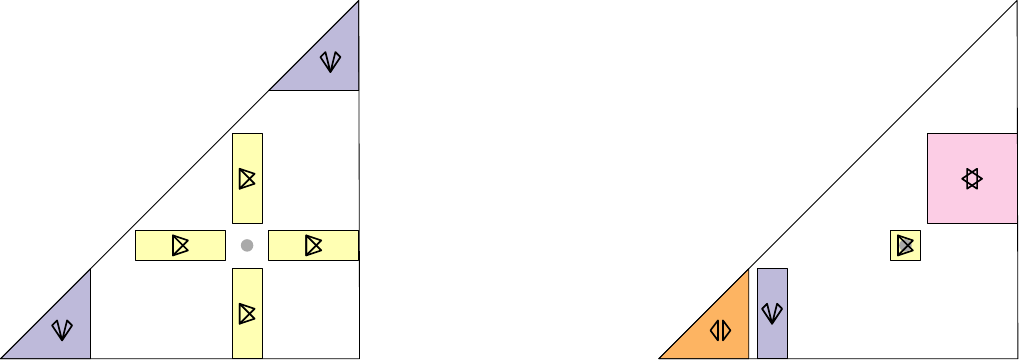}
   \hspace{2cm}\includegraphics[height=3cm]{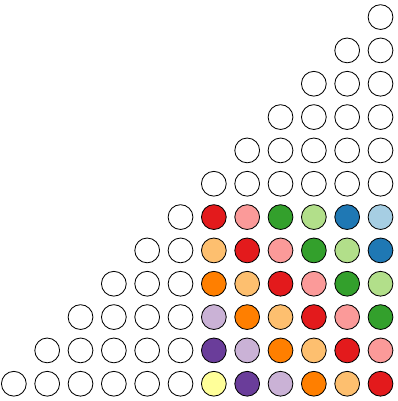}
}
For each $i\in\{1,\ldots,n/2\}$, we take $Q_i$ to be all the points of
$Q'=\{n/2,\ldots,n\}^2$ on the line $\{(x,y):y=3n/2-x-i+1\}$.
\end{proof}

\begin{thm}\thmlabel{swords-bat-ears-david}
  $\ex(n,\{\swords,\bat,\ears,\david\}) \in \Theta(n^2)$.
\end{thm}

\begin{proof}
\centerline{
   \includegraphics[height=3cm]{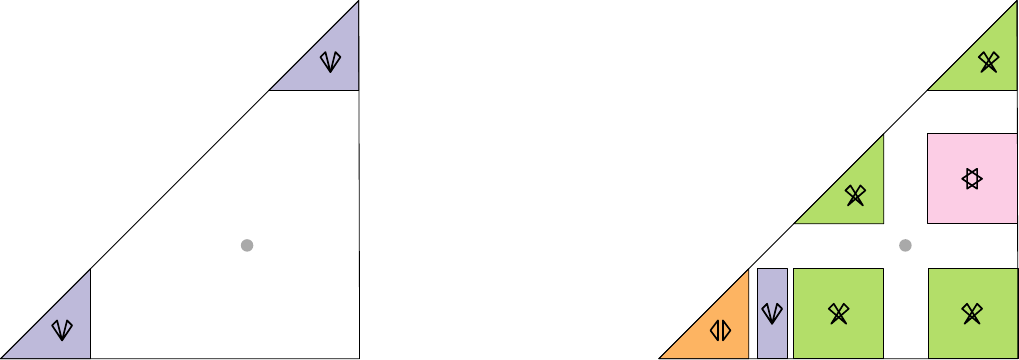}
   \hspace{2cm}\includegraphics[height=3cm]{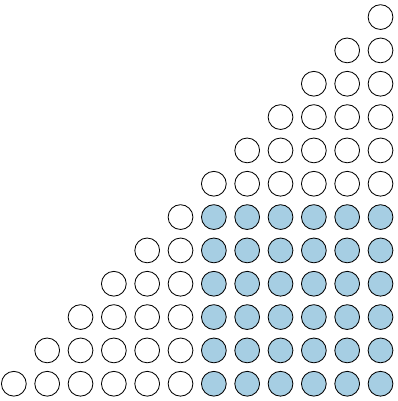}
}
   We take $Q_1=\{(x,y)\in Q: x>n/2, y<n/2\}$ and
   set $Q_2=Q_3=\cdots=Q_n=\emptyset$.
\end{proof}

\begin{thm}\thmlabel{david-nested-crossing}
  $\ex(n,\{\david,\nested,\crossing\}) \in \Theta(n^2)$.
\end{thm}

\begin{proof}
   \centerline{\includegraphics[height=3cm]{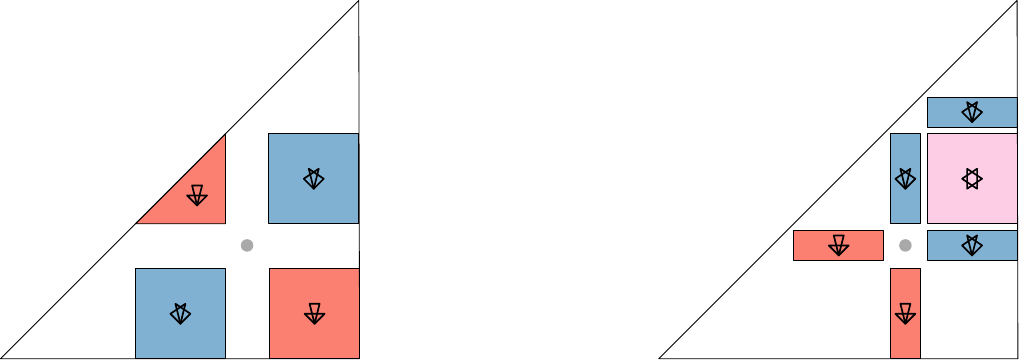}
   \hspace{2cm}\includegraphics[height=3cm]{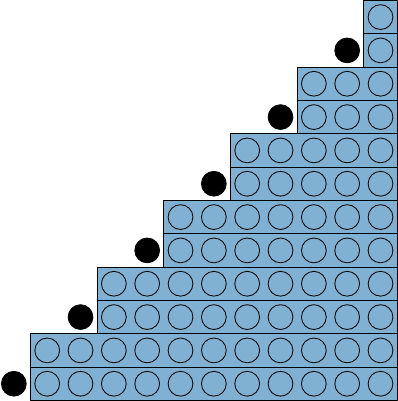}}
   We repeatedly take every second point on the diagonal $y=x-1$, i.e.,
   for each $i\in\{1,\ldots,n\}$, $Q_i=\{(2j,2j-1): j\in\{1,\ldots,n/2\}\}$.
\end{proof}

%
%

\begin{thm}\thmlabel{bat-nested-ears}
  $\ex(n,\{\bat,\nested,\ears\}) \in \Theta(n^2)$.
\end{thm}

\begin{proof}
   \centerline{\includegraphics[height=3cm]{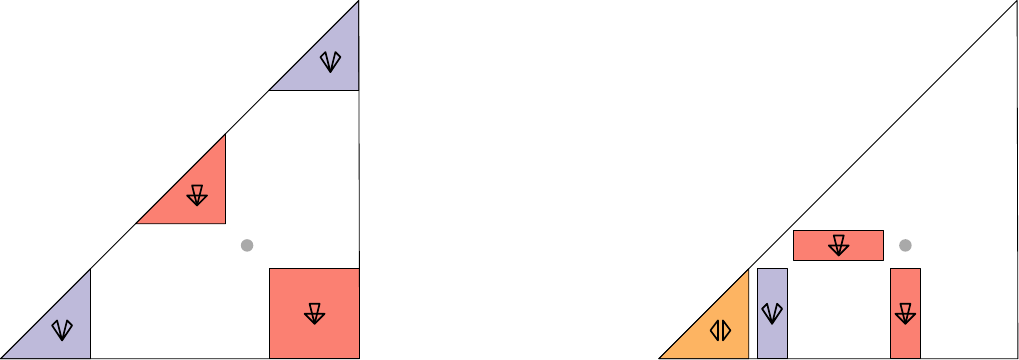}
   \hspace{2cm}\includegraphics[height=3cm]{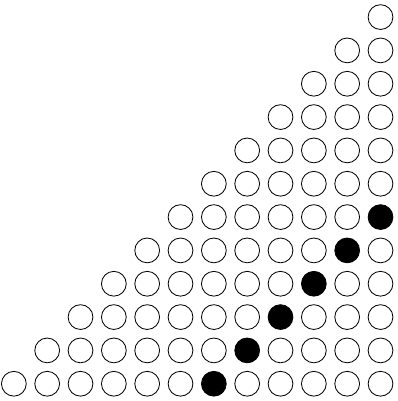}}
   We repeatedly take points on the diagonal of the lower-right quadrant,
   i.e., for each $i\in\{1,\ldots,n\}$, we take $Q_i=\{(n/2+i,i): i\in\{1,\ldots,n/2\}$.
\end{proof}

\section*{Acknowledgement}

Some of this work was carried out at the Fourth Annual Workshop on
Geometry and Graphs, held at the Bellairs Research Institute in Barbados,
March 6--11, 2016. The authors are grateful to the organizers and to
the other participants of this workshop for providing a stimulating working
environment.

\bibliographystyle{plainurl}
\bibliography{turan}

\end{document}